\documentclass[11pt,reqno]{amsart}
\usepackage[ams,nocomp,noannalen,nogrey,toc]{optional}

\usepackage{tabularx}
\usepackage{mathtools}
\opt{grey}{
\def\backgroundcolor{black!80}

\usepackage{xcolor}
\pagecolor{\backgroundcolor}
\color{white}
}
\opt{nogrey}{
\def\backgroundcolor{white}

}

\newcommand{\optbracket}[1]{\opt{ams}{#1}\opt{comp}{(#1)}}

\usepackage{enumitem}
\setenumerate{noitemsep,topsep=3pt,parsep=2pt}
\setitemize{noitemsep,topsep=3pt,parsep=2pt}

\opt{ams}{
\addtolength{\hoffset}{-1.25cm} \addtolength{\textwidth}{2.5cm}
\addtolength{\voffset}{-1.25cm} \addtolength{\textheight}{2cm}
}

\usepackage{amssymb, amsmath, amsthm, amsfonts, comment, subcaption}
\usepackage{bbm}
\usepackage[colorlinks=false]{hyperref}
\usepackage{tikz-cd}
\usetikzlibrary{matrix,arrows,decorations,decorations.markings,calc,positioning,intersections}
\usepackage{float}
\usepackage{xcolor}

\usepackage{setspace}
\opt{noams}{}
\opt{ams}{}

\tikzset{every picture/.style={scale=2,line width=0.7pt}}
\tikzstyle{longdashed}=[dashed,dash pattern=on 6pt off 6pt]
\tikzset{gap/.style={inner sep=2pt,fill=\backgroundcolor}}

\newcommand\stdscale{1.1}
\newcommand\stdshrink{0.95}
\newcommand\stdstyle{\footnotesize}
\newcommand\compstyle{\small}

\newcommand\stdskip\medskip 
\newcommand\hrefsf[2]{\href{#1}{\textsf{#2}}}
\newcommand\captionsized[1]{\caption{\opt{ams}{\footnotesize} #1}}

\newcommand\compdiagscale{1.1}

\newcommand{\equalsSep}{0.8pt}

\newcommand{\into}{\hookrightarrow}
\newcommand{\isoto}{\xrightarrow{\sim}} 

\def\arrowSep{2.25pt}
\newcommand\quoted[1]{`#1'}

\newcommand\Z{\mathbb Z}

\newcommand\A{\mathbb A}

\newcommand \placeholder{-}

\newcommand\id{\operatorname{id}}

\DeclareMathOperator \Hom {Hom}
\newcommand\sHom{\mathcal{H}om }

\newcommand\RDerived{\mathrm{R}}

\newcommand{\Spec}{\operatorname{Spec}}\newcommand{\Tot}{\operatorname{Tot}}

\newcommand{\comp}{} 

\newcommand\Cone{\mathsf{Cone}}

\newcommand\D{\mathsf{D}}

\newcommand\Exc{\mathrm{E}}
\newcommand\Bl {\mathrm{Bl}}
\newcommand\normal[1] {\mathcal{N}_{#1}}
\newcommand\normalof[2] {\mathcal{N}_{#1} {#2}}
\newcommand\conormalof[2] {\mathcal{N}^\vee_{#1} {#2}}

\newcommand\normalcone[2]{\mathcal{C}_{#1} {#2}}
\newcommand\zeroes[1] {V(#1)}
\newcommand\ideal {\mathcal{I}}

\newcommand\Ladj[1]{{#1}^{\mathsf{L}}}
\newcommand\Radj[1]{{#1}^{\mathsf{R}}}

\newcommand\codim{\operatorname{codim}}

\newcommand\rk{\operatorname{rk}}

\renewcommand\P{\mathbb P}

\newcommand\pt{\operatorname{pt}}
\newcommand\inc{\mathit{inc}}

\newcommand\cL{\mathcal L}
\newcommand\cM{\mathcal M}

\newcommand{\cE}{\mathcal{E}}
\newcommand{\cF}{\mathcal{F}}

\newcommand{\cO}{\mathcal{O}}

\newlength\tempWidth
\newcommand\InSpaceOf[2]{
   \settowidth{\tempWidth}{$#1$}\phantom{#1}\hspace{-\tempWidth}{#2}}

\newcounter{keyenvcountA}
\newcounter{keyenvcountB}
\newcounter{keyenvcountC}

\numberwithin{equation}{section}

\def\beq{\begin{equation*}}
\def\eeq{\end{equation*}}
\def\beqn{\begin{equation}}
\def\eeqn{\end{equation}}

\theoremstyle{plain}
\newtheorem{thm}{Theorem}[section]

\newtheorem{keythm}[keyenvcountB]{Theorem}
\newtheorem{prop}[thm]{Proposition}
\newtheorem{lem}[thm]{Lemma}
\newtheorem{cor}[thm]{Corollary}

\theoremstyle{remark}
\newtheorem{rem}[thm]{Remark}

\newcommand\acksize{\opt{ams}{\footnotesize}} 
\newtheorem*{acks}{\acksize Acknowledgements}

\theoremstyle{definition}

\newtheorem{defn}[thm]{Definition}

\newtheorem{setn}[thm]{Setting}
\newtheorem{keyassm}[keyenvcountC]{Assumption}

\newtheorem*{keynotn}{Notation} 

\newtheorem{eg}[thm]{Example}
\newtheorem{assm}[thm]{Assumption}

\newcommand{\firstproofpart}[1]{\noindent {(#1)}}
\newcommand{\proofpart}[1]{\stdskip \noindent {(#1)}}

\newcommand{\marginparstretch}{0.6}
\let\oldmarginpar\marginpar
\renewcommand\marginpar[1]{\-\oldmarginpar[\framebox{\setstretch{\marginparstretch}\begin{minipage}{\marginparwidth}{\raggedleft\scriptsize #1}\end{minipage}}]{\framebox{\setstretch{\marginparstretch}\begin{minipage}{\marginparwidth}{\raggedright\scriptsize #1}\end{minipage}}}}

\def\baseheight{0.25}
\def\baseskew{0.22}
\def\basetilt{0.06}
\def\basebend{12}
\def\baseshift{-0.09}
\def\baseperspective{0.9}

\def\fibwidth{0.24}
\def\fibheight{0.68}
\def\fibgrow{0.10}
\def\fibskew{-0.35}
\def\fibsep{0.4}
\def\dotsize{0.3pt}
\def\coneRskew{0.19}
\def\coneLskew{-0.20}
\def\coneext{1.1}
\def\conethick{1.5pt}
\def\bgcol{gray!30}
\def\labelsepn{0.06}
\def\labelpos{0.7}
\def\arrowshrink{0.05}
\def\curveheight{0.13}
\def\curvebend{35}
\def\curveshift{0.1}
\def\curvetweakX{0.08}
\def\curvetweakY{0.05}
\def\resgrow{0.7}
\def\resfibsep{0.02}

\def\flagSingular{0}
\def\flagResolution{1}

\def\draft{0}
\def\flagDraft{1}

\newcommand{\curve}[5]{
\draw[name path=#1,#2] ($(#3)+(#4,-\curveheight+#5)$) to [bend left=\curvebend] ($(#3)+(0,+\curveheight)$);
}

\newcommand{\fibcoord}[5]{
\coordinate (#1TL) at (#2-\fibwidth,#3+\fibheight+\fibskew+#4);
\coordinate (#1TR) at (#2+\fibwidth,#3+\fibheight-\fibskew+#4);
\coordinate (#1BL) at (#2-\fibwidth,#3-\fibheight+\fibskew-#5);
\coordinate (#1BR) at (#2+\fibwidth,#3-\fibheight-\fibskew-#5);
}

\newcommand{\threefoldpicture}[1]{
\def\flag{#1}
\begin{tikzpicture}[scale=2.0]

\coordinate (baseTL) at (-\baseperspective+\baseskew*\baseperspective,\baseheight+\baseshift-\basetilt);
\coordinate (baseTR) at (+\baseperspective+\baseskew*\baseperspective,\baseheight+\baseshift+\basetilt);
\coordinate (baseBL) at (-1-\baseskew,-\baseheight+\baseshift-\basetilt);
\coordinate (baseBR) at (+1-\baseskew,-\baseheight+\baseshift+\basetilt);

\ifx\flag\flagResolution
\curve{curveTR}{}{baseTR}{0}{0};
\curve{curveBL}{}{baseBL}{\curvetweakX}{\curvetweakY};
\curve{curveBR}{}{baseBR}{0}{0};
\curve{curveR}{}{+\fibsep+\resfibsep,+\curveshift}{0}{0};
\curve{curveL}{}{-\fibsep-\resfibsep,+\curveshift}{0}{0};
\curve{curve}{line width=\conethick}{0,+\curveshift}{0}{0};
\fi

\ifx\flag\flagSingular
\draw (baseBL) to [bend left=\basebend] (baseTL) to [bend left=\basebend] (baseTR) to [bend right=\basebend] (baseBR);
\fi 

\ifx\flag\flagResolution
\draw ($(baseBL)+(0,+\curveheight)$) to [bend left=\basebend] ($(baseTL)+(0,+\curveheight)$) to [bend left=\basebend] ($(baseTR)+(0,+\curveheight)$) to [bend right=\basebend] ($(baseBR)+(0,+\curveheight)$);
\draw ($(baseTR)+(0,-\curveheight)$) to [bend right=\basebend] ($(baseBR)+(0,-\curveheight)$);
\draw ($(baseBR)+(0,-\curveheight)$) [bend right=\basebend] ($(baseBL)+(\curvetweakX,-\curveheight+\curvetweakY)$);
\fi

\ifx\flag\flagResolution
\fibcoord{fib}{0}{0}{\resgrow*\curveheight}{\fibgrow}
\fibcoord{Lfib}{-\fibsep}{0}{\resgrow*\curveheight}{\fibgrow}
\fibcoord{Rfib}{+\fibsep}{0}{\resgrow*\curveheight}{\fibgrow}
\else
\fibcoord{fib}{0}{0}{0}{0}
\fibcoord{Lfib}{-\fibsep}{0}{0}{0}
\fibcoord{Rfib}{+\fibsep}{0}{0}{0}
\fi

\draw[\bgcol] (LfibTL) -- (LfibTR) --(LfibBR) --  (LfibBL) --  cycle;

\draw[\bgcol,line width=\conethick] (fibTL) -- (fibTR) --(fibBR) --  (fibBL) --  cycle;

\draw (RfibTL) -- (RfibTR) --(RfibBR) --  (RfibBL) --  cycle;

\path [name path=fibTR--fibBR] (fibTR) -- (fibBR);
\path [name path=LfibTL--LfibTR] (LfibTL) -- (LfibTR);
\path [name path=RfibTL--RfibTR] (RfibTL) -- (RfibTR);
\path [name path=fibBL--fibBR] (fibBL) -- (fibBR);
\path [name path=RfibBL--RfibBR] (RfibBL) -- (RfibBR);
\path [name path=RfibTL--RfibBL] (RfibTL) -- (RfibBL);

\path [name path=LfibTR--LfibBR] (LfibTR) -- (LfibBR);
\path [name path=fibTL--fibTR] (fibTL) -- (fibTR);
\path [name path=LfibBL--LfibBR] (LfibBL) -- (LfibBR);
\path [name path=fibTL--fibBL] (fibTL) -- (fibBL);

\path [name intersections={of=fibTR--fibBR and RfibTL--RfibTR,by=RfibTint}];
\path [name intersections={of=fibBL--fibBR and RfibTL--RfibBL,by=RfibBint}];

\path [name intersections={of=LfibTR--LfibBR and fibTL--fibTR,by=LfibTint}];
\path [name intersections={of=LfibBL--LfibBR and fibTL--fibBL,by=LfibBint}];

\draw (LfibTint) -- (LfibTR) -- (LfibTL) -- (LfibBL) -- (LfibBint);
\draw[line width=\conethick] (RfibTint) -- (fibTR) -- (fibTL) -- (fibBL) -- (RfibBint);

\ifx\flag\flagSingular
\draw[fill=black] (0,0) circle (\dotsize);
\draw[fill=black] (+\fibsep,0) circle (\dotsize);
\draw[fill=black] (-\fibsep,0) circle (\dotsize);
 \node (labelz) at (-\labelsepn*\labelpos,+\labelsepn) {$\InSpaceOf{y}{z}$};
 \node (labely) at (+\fibsep-\labelsepn*\labelpos,+\labelsepn) {$y$};
\fi

\ifx\flag\flagResolution
 \node (labelEz) at (-\labelsepn,\curveheight+\labelsepn+\curveshift) {$\excBlowupSingularShort_{z\InSpaceOf{y}{}}$};
 \node (labelEy) at (-\labelsepn+\fibsep,\curveheight+\labelsepn+\curveshift) {$\excBlowupSingularShort_y$};
\fi

\coordinate (LconeT) at (-\fibsep+\coneext*\coneLskew,\coneext);
\coordinate (LconeB) at (-\fibsep-\coneext*\coneLskew,-\coneext);
\coordinate (RconeT) at (+\fibsep+\coneext*\coneRskew,\coneext);
\coordinate (RconeB) at (+\fibsep-\coneext*\coneRskew,-\coneext);

\path [name path=RconeT--RconeB] (RconeT) -- (RconeB);
\path [name path=LconeT--LconeB] (LconeT) -- (LconeB);

\path [name intersections={of=LconeT--LconeB and LfibTL--LfibTR,by=LconeTint}];
\path [name intersections={of=LconeT--LconeB and LfibBL--LfibBR,by=LconeBint}];
\path [name intersections={of=RconeT--RconeB and RfibTL--RfibTR,by=RconeTint}];
\path [name intersections={of=RconeT--RconeB and RfibBL--RfibBR,by=RconeBint}];

\ifx\draft\flagDraft
\draw[\bgcol] (LconeT) -- (LconeB);
\draw[\bgcol] (RconeT) -- (RconeB);
\fi

\draw[line width=\conethick] (LconeTint) -- (LconeBint);
\draw[line width=\conethick] (RconeTint) -- (RconeBint);

\ifx\flag\flagResolution
\path [name intersections={of=LconeT--LconeB and curveL,by=Ltransint}];
\path [name intersections={of=RconeT--RconeB and curveR,by=Rtransint}];
\draw[fill=black] (Ltransint) circle (\dotsize);
\draw[fill=black] (Rtransint) circle (\dotsize);
\fi

\ifx\flag\flagSingular
\draw (baseBR) to [bend right=\basebend]  (baseBL);
\fi 

\ifx\flag\flagResolution
\draw ($(baseBR)+(0,+\curveheight)$) to [bend right=\basebend] ($(baseBL)+(0,+\curveheight)$);
\draw ($(baseBR)+(0,-\curveheight)$) to [bend right=\basebend] ($(baseBL)+(\curvetweakX,-\curveheight+\curvetweakY)$);
\fi

\ifx\flag\flagSingular
\coordinate (labelBpos) at (+1-\labelsepn,-\baseheight);
\else
\coordinate (labelBpos) at (+1,-\baseheight-2*\curveheight);
\fi

\ifx\flag\flagSingular
 \node (labelY) at (labelBpos) {$\blowupLocusA$};
\fi
\ifx\flag\flagResolution
 \node (labelY) at ($(labelBpos)+(0,+\curveheight)$) {
 $\excBlowupAmbient$
 };
\fi

 \ifx\flag\flagSingular
 \coordinate (labelNpos) at (-\fibsep-\fibwidth/3,\fibheight+3*\labelsepn);
 \node (labelN) at (labelNpos) {$\normalof{\blowupLocusA}{\ambient}$};
 \fi
\ifx\flag\flagResolution
 \coordinate (labelNpos) at (-\fibsep-\fibwidth/3,\fibheight+5*\labelsepn);
 \node (labelN) at (labelNpos) {$\normalof{\excBlowupSingularShort}{\resolutionAmbient}$};
 \fi

\coordinate (labelCpos) at (\fibsep+\fibwidth+4*\labelsepn,\fibheight-3*\labelsepn);
\path [name path=labelCposL] (labelCpos) --+ (-1,0);
\path [name intersections={of=labelCposL and RconeT--RconeB,by=labelCposint}];

\ifx\flag\flagSingular
\node (labelC) at (labelCpos) {$\normalcone{\blowupLocusA}{\singular}$};
\draw[-{Latex[round]}] (labelC) -- ($(labelCposint)+(\arrowshrink,0)$);
\fi

\end{tikzpicture}
}

\def\surfheight{2}
\def\ellthick{1.0pt}
\def\ellheight{0.4}
\def\singtweakx{0.03}
\def\singtweaky{0.1}
\def\resbend{25}
\def\gridx{12}
\def\gridy{8}
\def\boxsize{2}
\def\boxstretch{1.1}
\def\boxskew{0.7}
\def\boxdepth{0.8}
\def\boxshift{0.2}
\def\apexsize{1.5pt}
\def\excradius{0.75}

\def\exctweakx{0}
\def\exctweaky{0.05}
\def\surflabelsepn{0.5}
\def\surfarrowshrink{0.2}
\def\commscale{2}
\def\commshift{-0.5}

\newcommand{\surfbox}[2]{

\coordinate (TBL) at (#1-\boxsize*\boxstretch+\boxskew*\boxdepth,#2+\boxdepth+\boxsize);
\coordinate (TBR) at (#1+\boxsize*\boxstretch+\boxskew*\boxdepth,#2+\boxdepth+\boxsize);
\coordinate (TFR) at (#1+\boxsize*\boxstretch-\boxskew*\boxdepth,#2-\boxdepth+\boxsize);
\coordinate (TFL) at (#1-\boxsize*\boxstretch-\boxskew*\boxdepth,#2-\boxdepth+\boxsize);
\coordinate (BBL) at (#1-\boxsize*\boxstretch+\boxskew*\boxdepth,#2+\boxdepth-\boxsize);
\coordinate (BBR) at (#1+\boxsize*\boxstretch+\boxskew*\boxdepth,#2+\boxdepth-\boxsize);
\coordinate (BFR) at (#1+\boxsize*\boxstretch-\boxskew*\boxdepth,#2-\boxdepth-\boxsize);
\coordinate (BFL) at (#1-\boxsize*\boxstretch-\boxskew*\boxdepth,#2-\boxdepth-\boxsize);

\draw (TBL) -- (TBR)  --  (TFR) -- (TFL)  -- cycle;
\draw  (TFL)  -- (BFL) -- (BFR) -- (TFR);
\draw (BFR) --  (BBR) -- (TBR);

}

\newcommand{\surfcone}[4]{
\def\flag{#1}

\draw[#4,line width=\ellthick] (#2,#3+\surfheight) ellipse (1 and \ellheight);
\draw[#4,line width=\ellthick] (#2,#3-\surfheight) ellipse (1 and \ellheight);

\ifx\flag\flagResolution

\draw[#4,name path=resR] (#2+1-\singtweakx,#3+\surfheight-\singtweaky) to [bend right=\resbend] (#2+1-\singtweakx,#3-\surfheight+\singtweaky);
\draw[#4,name path=resL] (#2-1+\singtweakx,#3+\surfheight-\singtweaky) to [bend left=\resbend] (#2-1+\singtweakx,#3-\surfheight+\singtweaky);

\path[\bgcol,name path=exc]  (#2+3/2,#3+2*\surfheight/3) to  (#2-3/2,#3-2*\surfheight/3);
\path [name intersections={of=resL and exc,by=excL}];
\path [name intersections={of=resR and exc,by=excR}];
\draw[#4] (excL) to [bend right=\resbend] (excR);

\else

\draw[#4] (#2+1-\singtweakx,#3+\surfheight-\singtweaky) -- (#2-1+\singtweakx,#3-\surfheight+\singtweaky);
\draw[#4] (#2-1+\singtweakx,#3+\surfheight-\singtweaky) -- (#2+1-\singtweakx,#3-\surfheight+\singtweaky);
\draw[#4,fill=#4] (#2,#3) circle (\apexsize);

\fi
}

\newcommand{\surfpicture}{
\begin{tikzpicture}[scale=0.3]

    	\node (singular) at (\gridx/2-\commscale+\commshift,\gridy/2-\commscale) {$\singular$};
	\node (resolution) at (\gridx/2-\commscale+\commshift,\gridy/2+\commscale) {$\resolution$}; 
	\node (ambient) at (\gridx/2+\commscale+\commshift,\gridy/2-\commscale) {$\ambient$};
	\node (blowupAmbient) at (\gridx/2+\commscale+\commshift,\gridy/2+\commscale) {$\resolutionAmbient$};
	\draw [right hook->] (singular) to  (ambient);
	\draw [->] (resolution) to node[left]{\stdstyle $\resMap$} (singular);
	\draw [right hook->] (resolution) to (blowupAmbient);
	\draw [->] (blowupAmbient) to  (ambient);
	\draw [->] (resolution) to  node[above right]{\stdstyle $\compMap$} (ambient);

\surfcone{1}{0}{\gridy}{}
\surfcone{1}{\gridx}{\gridy}{\bgcol}
\surfcone{0}{0}{0}{}
\surfcone{0}{\gridx}{0}{\bgcol}

\surfbox{\gridx+\boxshift}{\gridy}
\surfbox{\gridx+\boxshift}{0}

\draw (\gridx+\exctweakx,\gridy+\exctweaky) circle (\excradius);

\coordinate (labelYpos) at (-5*\surflabelsepn,0);
\node (labelY) at (labelYpos) {$\blowupLocusA$};
\draw[-{Latex[round]}] (labelY.east) -- (-\surflabelsepn,0);

\coordinate (labelEXpos) at (-1-3*\surflabelsepn,\gridy);
\node (labelEX) at (labelEXpos) {$\excBlowupSingular$};
\draw[-{Latex[round]}] (labelEX.east) -- (0,\gridy);

\coordinate (labelESpos) at (\gridx+\boxsize+\boxsize*\boxstretch+\surflabelsepn/2,\gridy);
\node (labelES) at (labelESpos) {$\excBlowupAmbient$};
\draw[-{Latex[round]}] (labelES.west) -- (\gridx+\excradius+\exctweakx+\surfarrowshrink,\gridy);
\end{tikzpicture}

}

\begin{document}

\title[Derived symmetries for crepant contractions]{\opt{ams}{\phantom{.}\\[-20pt] }Derived symmetries for crepant contractions \\ to hypersurfaces}
\author{W.\ Donovan}
\opt{comp}{\email{donovan@mail.tsinghua.edu.cn}}
\address{Yau MSC, Tsinghua University, Haidian, Beijing, China; BIMSA, Yanqi Lake, Huairou, Beijing, China; Kavli IPMU (WPI), TODIAS, University of Tokyo, Kashiwanoha, Chiba, Japan}
\opt{ams}{\email{donovan@mail.tsinghua.edu.cn}}

\thanks{I am supported by Yau MSC, Tsinghua University, BIMSA, and China TTP}

\begin{abstract} 
Given a crepant contraction~$f$ to a singularity~$X$, we may expect a derived symmetry of the source of~$f$. Under easily-checked geometric assumptions, I~construct such a symmetry when~$X$ is a hypersurface in a smooth ambient~$S$, using a spherical functor from the derived category of~$S$. I~describe this symmetry, relate it to other symmetries, and establish its compatibility with base change.
\end{abstract}



\maketitle


\def\biggerscale{1.2}
\def\diagscale{0.85}


\def\field{\mathbbm{k}}
\def\blowupLocusA{Y}
\def\blowupLocusAcoord{y}
\def\resolution{{\tilde{\singular}}}
\def\resolutionAmbient{{\tilde{\ambient}}}
\def\singular{X}
\def\ambient{{S}} \def\ambientbc{{\psi}} \def\ambientSource{{R}} 
\def\ambientSourcebc{{\phi}} 
\def\singembed{{\tau}}
\def\bcembed{{\ambientSourcebc}} 
\def\blowupSingular{\Bl_{\blowupLocusA} \singular}
\def\excBlowupSingular{\Exc_{\blowupLocusA} \singular}
\def\excBlowupSingularTilde{\Exc_{\blowupLocusA} \resolution} 
\def\excBlowupSingularShort{\Exc}
\def\blowupAmbient{\Bl_{\blowupLocusA} \ambient}
\def\excBlowupAmbient{\Exc_{\blowupLocusA} \ambient}
\def\blowupLocusB{Z}
\def\singLocus{\mathrm{Sing}_\singular}
\def\blowupB{\Bl_{\blowupLocusB} \blowupLocusA}
\def\excBlowupB{\Exc_{\blowupLocusB} \blowupLocusA}
\def\excBlowupBshort{\Exc}

\def\resMap{f}
\def\compMap{g}
\def\blowupAmbientMap{h}
\def\incInAmbient{i}
\def\incResInBlowupAmbient{\inc}
\def\incExcInExcAmbient{\inc'}
\newcommand\incExcInResolution[1][]{{j}_{\blowupLocusA#1}'}
\def\blowupProjAmbient{{r}}
\def\blowupProjLocal{{s}}
\def\blowupIncAmbient{i_{\blowupLocusA}}
\def\blowupBInc{{i_{\blowupLocusB}}}
\newcommand\blowupExcIncAmbient[1][]{j_{\blowupLocusA#1}}
\newcommand\blowupExcBInc[1][]{j_{\blowupLocusB#1}}
\def\blowupProjA{p}
\def\blowupProjB{q}
\def\SODembedA{\Phi}
\def\SODembedB{\Psi}

\def\divcasetwist{\cM}

\def\gendivisor{D}
\def\gendivisorambient{X}

\def\pagoda{w}
\def\index{m}

\newcommand\fun[1][F]{\operatorname{\mathsf{#1}}}

\newcommand{\twObj}[2]{\fun[T]_{#2}} 
\newcommand{\twk}[1][]{\fun[S]_{#1}}
\newcommand{\hypertw}[1][]{\fun[hyperT]_{#1}}
\newcommand{\nctw}{\fun[T]_{\mathsf{nc}}}
\newcommand{\tw}[1][]{\fun[T]_{#1}}
\newcommand{\lb}[1][]{\fun[L]_{#1}}
\newcommand{\ctw}[1][]{\fun[C]_{#1}}
\newcommand{\twd}[1][]{\fun[T]'_{#1}}
\newcommand{\ctwd}[1][]{\fun[C]'_{#1}}

\def\funFlopFlop{\fun[FlopFlop]}
\def\gun{\fun[G]}
\def\hun{\fun[H]}
\def\hunb{\fun[H]_1}
\def\iun{\fun[J]}

\newcommand\serre[1]{\fun[S]}

\def\exbase{\blowupLocusA}
\def\exbasept{z}
\def\bunsec{\cF}
\def\bunseccoord{w}
\def\buntot{\cF'}
\def\buntotcoord{z}
\def\buntwist{\divcasetwist_\blowupLocusA}
\def\bunsecproj{s}
\def\buntotproj{s'}
\def\sec{\theta}
\def\excsec{\rho}
\def\blowsec{\sigma}
\def\bunproj{\pi}
\def\bunprojres{\tilde{\pi}}
\def\bunlin{\cL}
\def\buntheta{\Theta}
\def\indsec{\eta}


\opt{toc}
{
\setcounter{tocdepth}{1}
{\small\tableofcontents}
}

\section{Introduction}

\subsection*{Background}
The derived category of coherent sheaves on a variety is a fundamental invariant. In recent years it has continued to be a key tool in subjects including birational and enumerative geometry, mirror symmetry, and moduli, see for instance~\cite{BM,HLDequivalence,PT,SS1,TodDTPT}. Its autoequivalences may be considered as \emph{derived symmetries} of the variety. 
\opt{ams}{

}Derived symmetries have found important and diverse applications, in particular in studying symplectic geometry via mirror symmetry~\cite{KSmith,SS2},  
in constraining enumerative geometry~\cite{BMX,TodGV,TodspP2}, 
and in categorifying perverse sheaves~\cite{KS1,BKS}.

\opt{ams}{\stdskip}

The derived symmetry group of a variety depends strongly on its canonical bundle. Varieties with ample or antiample canonical bundle have tightly constrained derived symmetries~\cite{BO}, whereas varieties with trivial canonical bundle may have much richer derived symmetry groups. As a relative analog of this, a class of birational morphisms $\resolution \to \singular$ with trivial {relative} canonical bundle, known as \emph{crepant contractions}, may be expected to yield derived symmetries of $\resolution$.

\subsection*{Existing work}

Derived symmetries of such $\resolution$ have been constructed and studied for various classes of singularities $\singular$, for instance for Du Val surface singularities~\cite{ST} and compound Du Val singularities of 3-folds~\cite{DW1,DW3}. There has also been work on higher-dimensional examples and classes of hyperk\"ahler~$\resolution$~\cite{CKL,DS1,HuyTho}.
\opt{ams}{

}These approaches are often restricted to particular geometries, or require increasingly challenging homological calculations in higher dimensions. Furthermore, global assumptions such as (quasi)projectivity are commonly required to control coherent duality.

Grade restriction window techniques from~\cite{BFK,HL} are effective for constructing derived symmetries especially when the contraction is associated to a `balanced' wall crossing~\cite{HLShi}. Noncommutative methods have led to general results in particular for \mbox{3-folds} and contractions of families of curves~\cite{BB,DW1,DW4,IW9}, though these methods face challenges in higher dimensions from obstructions to noncommutative crepant resolutions~\cite{Dao}.
\opt{ams}{

}The author and Wemyss gave a derived symmetry construction for general fibre dimension in~\cite{DW4}, assuming a relative tilting bundle and using the associated sheaf of noncommutative algebras. Barbacovi gave a different general construction in a flop setting~\cite{Bar}, generalizing results from~\cite{BB}.

\opt{ams}{\stdskip}

Nevertheless, for general contractions with fibres of dimension more than~$1$, constructing and understanding derived symmetries remains difficult. Addressing this challenge may be seen as complementary to ongoing efforts to address the long-standing D-equivalence conjecture~\cite{Kawamata}. It is therefore important to develop new approaches to drive progress, especially in higher dimensions, and to develop general constructions of derived symmetries using minimal assumptions.

\subsection*{This paper} For a large class of crepant contractions $\resolution \to \singular$, I construct an  associated derived symmetry of~$\resolution$ using a smooth ambient space for~$\singular$. Taking the case when~$\singular$ is a hypersurface, this approach yields new derived symmetries and relations between them in arbitrary dimension, as well as new descriptions of existing equivalences. 
\opt{ams}{

}My standing assumptions in Section~\ref{sect.setting} are geometric and easily-checked. Apart from assuming some equidimensionality, these assumptions are furthermore {local} on~$\singular$, so that in particular projectivity or quasi-projectivity are not needed for the main results.\footnote{More restrictive assumptions are used in Theorem~\ref{thm.mainintroB} to describe the derived symmetry in a class of contractions with higher fibre dimension.\opt{ams}{ These assumptions imply $\resolution$ and $\singular$ are projective Calabi--Yau.}}

I achieve this by exploiting the interplay between Orlov's blowup formula for the derived category, and known spherical functors for Cartier divisors. Strikingly, no explicit Ext calculations are required --- I do not even need to consider particular sheaves on spaces, apart from those associated to the exceptional divisors of blowups, and their tensor powers.

\subsection{Setting}\label{sect.setting} Take a birational contraction of reduced separated schemes
\begin{equation*}
\resMap\colon\resolution\to\singular
\end{equation*}
where `contraction' means that $\resMap$ is projective, and the  
(non-derived) pushforward of $\cO_{\resolution}$ is~$\cO_{\singular}$. I will use the following assumptions throughout.\footnote{As a mnemonic, we have `a' for ambient, `b' for blowup, and `c' for crepant.}

\begin{keyassm}\label{assm.embed}   $\singular$ is a hypersurface in a smooth equidimensional ambient~$\ambient$.\end{keyassm}

Our interest will be in the case when $\singular$ is singular. When $\singular$ is cut out by a global function\opt{ams}{ }\opt{comp}{~}$t$, we may think of $t\colon \ambient \to \mathbb{A}^1$ as a smoothing of $\singular$.

\begin{keyassm}\label{assm.blowup} $\resMap$ is a blowup along a smooth subscheme~$\blowupLocusA$ of constant codimension~$n$.
\end{keyassm}

\begin{keyassm}\label{assm.crep} $\resMap$ is crepant.  \end{keyassm}

Here \quoted{crepant} means that $\omega_\resolution \cong \resMap^* \omega_\singular$ or equivalently that $\omega_\resMap = \omega_\resolution \otimes \resMap^* \omega_\singular^\vee$ is trivial.\footnote{Here $\resolution$ and $\singular$ are Gorenstein (see Proposition~\ref{prop.smCar2}) so their canonical sheaves are invertible.} We say \quoted{crepant} following Reid because $\omega_\resMap$ measures \quoted{discrepancy}. 

\opt{ams}{\stdskip}

First examples of this setting may be obtained by letting $\singular$ be the affine cone over a hypersurface of degree $n$ in $\mathbb{P}^n$, and $\resMap$ the blowup at its vertex, see Example~\ref{eg.cone}.

\begin{keynotn} Let
$\compMap\colon\resolution\to\ambient$
be $\resMap$ composed with the inclusion of~$\singular$ into~$\ambient$.
\end{keynotn}

\begin{rem} For $n\geq 2$, our assumptions imply $\singular$ singular: otherwise, $\omega_\resMap$~corresponds to the exceptional divisor of~$\resMap$ with multiplicity $n-1>0$, contradicting crepancy. For $n=1$, so that $\blowupLocusA$ is a divisor in $\singular$, we will see that  interesting examples occur with $\blowupLocusA$ non-Cartier. 
\end{rem}

\subsection{Pullback from ambient} 

I first show that there is a strong homological relationship between~$\resolution$ and the ambient space~$\ambient$.

\begin{keythm}[\optbracket{Theorem~\ref{thm.main}}]\label{thm.pullbacksph} 
In the setting of Section~\ref{sect.setting} the pullback functor
\begin{equation*}\compMap^* \colon \D(\ambient) \to \D(\resolution)\end{equation*}
is spherical, where $\D$ denotes the bounded coherent derived category.
\end{keythm}

Spherical functors are briefly reviewed in Section~\ref{sec.sphfun}. They provide a natural language to describe and manipulate autoequivalences of (enhanced) triangulated categories and their relatives~\cite{AL,Seg}. For now, note that a spherical functor induces an autoequivalence of its target category: 
Theorem~\ref{thm.pullbacksph} thence gives a \emph{twist} autoequivalence~$\tw[\compMap^*]$ of $\D(\resolution)$ which fits into a triangle of Fourier--Mukai functors~\cite{Huy} as follows.
\begin{equation} \label{eq.twisttri}
 \compMap^* \compMap_* \to \id_{\resolution} \to \tw[\compMap^*] \to \end{equation}
The autoequivalence $\tw[\compMap^*]$ is central to this paper. I  will describe it in a range of natural geometric settings, and establish its properties.

\stdskip
I outline the proof of Theorem~\ref{thm.pullbacksph}. Let $\blowupAmbientMap\colon\resolutionAmbient\to\ambient$ be the blowup along $\blowupLocusA$, which is of constant codimension~$n+1$. Then Orlov~\cite{Orl} gives a semiorthogonal decomposition
\begin{equation}\label{eqn.keysod}
\begin{aligned}
\D(\resolutionAmbient) & = \big\langle \blowupAmbientMap^* \D(\ambient), \D(\blowupLocusA), \dots, \D(\blowupLocusA)\big\rangle
\end{aligned}
\end{equation}
with~$n$ copies of~$\D(\blowupLocusA)$. Writing~$\incResInBlowupAmbient\colon\resolution \into \resolutionAmbient$, it is known that the derived restriction functor $\incResInBlowupAmbient^*\colon\D(\resolutionAmbient)\to\D(\resolution)$ is spherical, using that $\incResInBlowupAmbient$ is the embedding of a Cartier divisor. Under a certain criterion of Halpern-Leistner and Shipman~\cite{HLShi} given in Section~\ref{sect.factoring}, the restriction of $\incResInBlowupAmbient^*$ to the subcategory $\blowupAmbientMap^* \D(\ambient)$ is also a spherical functor. I find that this criterion holds using the crepancy of $\resMap$, and conclude that $\incResInBlowupAmbient^*\blowupAmbientMap^*\cong\compMap^*$ is spherical.

\begin{rem} I do not require smoothness of $\resolution$, or that $\resMap$ is a morphism of normal schemes, see for instance Remark~\ref{eg.relaxed}.
\end{rem}

\subsection{Blowups in non-Cartier divisors} 

The following describes $\tw[\compMap^*]$  when $n=1$ in terms of the blowup of $\singular$ along $\blowupLocusA$.

\begin{keythm}[\optbracket{Theorem~\ref{thm.mainn1}}]\label{thm.mainintro} 
In the setting  of Section~\ref{sect.setting} suppose $n=1$.
Define a functor 
\begin{equation*}
	\begin{tikzpicture}[xscale=\compdiagscale]
		\node (A) at (-0.25,0) {$\gun\colon \D(\blowupLocusA)$};
		\node (B) at (1,0) {$\D(\excBlowupSingular)$};
		\node (C) at (2,0) {$\D(\resolution)$};

		\draw [->] (A) to node[above]{\compstyle $\blowupProjA^*$} (B);
		\draw [->] (B) to (C);
	\end{tikzpicture}
\end{equation*}
by composition, where $\blowupProjA$ is the restriction of the contraction $\resMap$ to the exceptional locus~$\excBlowupSingular$, and the last functor is  pushforward. Then:
\begin{enumerate}[label={(\arabic*)},ref={\arabic*}]
\item\label{thm.mainintropartA} $\gun$ is spherical.
\item\label{thm.mainintropartB} There is an isomorphism \begin{equation*}
\tw[\compMap^*](\placeholder) \otimes \normal{\resolution} \cong \tw[\gun]^{-1}(\placeholder) [2]
\end{equation*}
between autoequivalences of $\D(\resolution)$.
\end{enumerate}
Here $\normal{\resolution}$ is the invertible normal sheaf of $\resolution$ in $\resolutionAmbient=\blowupAmbient$.
\end{keythm}

The proof is a continuation of the proof of Theorem~\ref{thm.pullbacksph} for the case $n=1$. The criterion of Halpern-Leistner--Shipman used there also gives that the restriction of $\incResInBlowupAmbient^*$ to the other subcategory appearing in~\eqref{eqn.keysod}, namely $\D(\blowupLocusA)$, is a spherical functor. Using a certain base change isomorphism\footnote{Here I need some smoothness on~$\ambient$, see Remark~\ref{rem.smoothness}.} I find that this restriction is isomorphic to~$\gun$, yielding part~(\ref{thm.mainintropartA}). The criterion furthermore gives a factorization of the twist $\tw[\incResInBlowupAmbient^*] \cong \placeholder\otimes  \normal{\resolution}^\vee[2]$ with factors $\tw[\compMap^*]$ and $\tw[\gun]$. This factorization then rearranges to give~(\ref{thm.mainintropartB}).

\stdskip

A first example is as follows.

\begin{eg}\label{eg.conifold} Take $\singular=\{ac+bd=0\}$ a hypersurface in $\ambient = \mathbbm{k}^4$ with an ordinary double point, and $\resolution$ a small resolution. The latter may be obtained by taking $\blowupLocusA=\{c,d=0\}$, which is Cartier in $\singular$ except at the origin. We have $\blowupLocusA\cong\mathbbm{k}^2$, and a calculation of the blowup shows that $\excBlowupSingular\cong\Bl_{\{0\}} \mathbbm{k}^2$. Part~(\ref{thm.mainintropartA}) above therefore gives a spherical functor
\begin{equation*}
	\begin{tikzpicture}[xscale=\compdiagscale]
		\node (A) at (-0.25,0) {$\gun\colon \D(\mathbbm{k}^2)$};
		\node (B) at (1,0) {$\D(\Bl_{\{0\}} \mathbbm{k}^2)$};
		\node (C) at (2,0) {$\D(\resolution)$};

		\draw [->] (A) to node[above]{\compstyle $\blowupProjA^*$} (B);
		\draw [->] (B) to (C);
	\end{tikzpicture}
\end{equation*}
to the resolution $\resolution$. Then part~(\ref{thm.mainintropartB}) says that the twist  $\tw[\gun]$ is inverse to $\tw[\compMap^*]$ up to homological shift and tensoring by an invertible sheaf.
\end{eg}

I am not aware of other work on the functor~$\tw[\gun]$, even in this first example.

\stdskip

The following theorem establishes a general setting where $\excBlowupSingular\cong\blowupB$ for some~$\blowupLocusB$, in particular generalizing Example~\ref{eg.conifold}. It gives a description of $\tw[\compMap^*]$ in terms of the blowup of $\blowupLocusA$ along~$\blowupLocusB$. 
I allow this~$\blowupLocusB$ to be singular and non-reduced.

\begin{keythm}[\optbracket{Corollary~\ref{cor.blowupB}}]\label{keythm.blowupB} In the setting of Section~\ref{sect.setting} suppose $n=1$, and furthermore:

\begin{enumerate}[label={(\roman*)},ref={\roman*}]
\item\label{cor.blowupBgeomAssmA} The normal cone $\normalcone{\blowupLocusA}{\singular}$ is cut out of \,$\normalof{\blowupLocusA}{\ambient}$ by fibrewise linear functions induced by a section~$\sec$ of a rank~$2$ locally free sheaf on $\blowupLocusA$, as explained in Assumption~\ref{assm.cutnormal}.
\item\label{cor.blowupBgeomAssmB} The section $\sec$ in part~{(\ref{cor.blowupBgeomAssmA})} is regular. 
\end{enumerate}

\noindent Let $\blowupLocusB$ be the  (possibly singular and non-reduced) zeroes of~$\sec$ in $\blowupLocusA$. Then:

\begin{enumerate}[start=0,label={(\arabic*)},ref={\arabic*}]
\item\label{keythm.blowupBA} The projection $\blowupProjA\colon\excBlowupSingular\to\blowupLocusA$ is the blowup of\, $\blowupLocusA$\! along $\blowupLocusB$.
\end{enumerate}
For this latter blowup, let $\blowupProjB$ be the projection from $\excBlowupB$. We may put
\begin{equation*}
	\begin{tikzpicture}[xscale=\compdiagscale]
		\node (A) at (-0.25,0) {$\hun\colon \D(\blowupLocusB)$};
		\node (B) at (1,0) {$\D(\excBlowupB)$};
		\node (C) at (2.75,0) {$\D(\excBlowupB)$};
		\node (D) at (3.75,0) {$\D(\resolution)$};

		\draw [->] (A) to node[above]{\compstyle $\blowupProjB^*$} (B);
		\draw [->] (B) to node[above]{\compstyle $\otimes \cO_{\blowupProjB}(-1)$} (C);
		\draw [->] (C) to (D);
	\end{tikzpicture}
\end{equation*}
where the last functor is  pushforward. Then furthermore:
\begin{enumerate}[resume*]
\item\label{keythm.blowupBB} $\hun$ is spherical.
\item\label{keythm.blowupBC} There is an isomorphism
\begin{equation*}
 \tw[\compMap^*](\placeholder)\otimes \resMap^* \normal{\singular}  \cong  \tw[\hun] (\placeholder) [2]
\end{equation*}
 between  autoequivalences of $\D(\resolution)$.
\end{enumerate}
Here $\normal{\singular}$ is the invertible normal sheaf of $\singular$ in $\ambient$.
\end{keythm}

If $n\coloneqq\codim_\singular \blowupLocusA=1$ then the fibre dimension of $\resMap$ is at most~$1$, see 
Proposition~\ref{prop.smCar2}(\ref{prop.smCarR}). Note that $\blowupLocusB$ is codimension~$2$  in $\blowupLocusA$ by~\textnormal{(\ref{cor.blowupBgeomAssmB})} above, and so $\blowupLocusB$ is codimension~$3$  in $\singular$. Hence in the setting of the above theorem, $\resMap$ is necessarily a small contraction.

\begin{rem} If $\blowupLocusB$ is a fattened point, $\tw[\hun]$ is the fat spherical twist of Toda~\cite{Tod}, see Example~\ref{eg.pagoda}. This gives a new proof that this twist is an autoequivalence, avoiding a delicate Ext calculation to verify the fat spherical condition.

More generally, $\tw[\hun]$ is related to the spherical fibrations of Anno and Logvinenko~\cite{ALorthog}, and twists by sheafy contraction algebras~\cite{DW4}.  However the methods here are quite different, and the isomorphism~(\ref{keythm.blowupBC}) is new, to my knowledge.

Bodzenta and Bondal~\cite{BB} consider a setting with more general flops of curves, though with $\ambient$ affine and $\normal{\singular}$ trivial. They study a relative of the triangle~\eqref{eq.twisttri}, see Remark~\ref{rem.BB}. They do not show that $\compMap^*$ is spherical, but their work suggests an isomorphism as in~(\ref{keythm.blowupBC}) above, modulo tensoring by an invertible sheaf. This was one of the inspirations for me to formulate and prove Theorem~\ref{keythm.blowupB}.\end{rem} 

\begin{eg}\label{eg.conifoldcont}  
Take again $\singular=\{ac+bd=0\}$ in~$\ambient = \mathbbm{k}^4$ with an ordinary double point, and $\blowupLocusA = \{c,d=0\} \cong \mathbbm{k}^2$ with coordinates~$(a,b)$. Then $\normalof{\blowupLocusA}{\ambient}$ has fibre coordinates $(c,d)$. For the regular section $\sec = (a,b)$ of the trivial rank~$2$ locally free sheaf on $\blowupLocusA$, the induced function on $\normalof{\blowupLocusA}{\ambient}$ is $ac+bd$, which clearly cuts out the normal cone $\normalcone{\blowupLocusA}{\singular}$. We are therefore in the setting of Theorem~\ref{keythm.blowupB}: full details are given in Section~\ref{sec.eg}.

Noting that $\blowupLocusB = \{ 0 \}$ and $\excBlowupB \cong \P^1$, the twist $\tw[\hun]$ here simplifies to a twist~$\twObj{\resolution}{\cE}$ by a spherical object $\cE\in \D(\resolution)$, see Section~\ref{sec.sphobj}. The theorem thence gives the following.
\begin{equation}\label{eq.3fold}
 \tw[\compMap^*]  \cong \twObj{\resolution}{\cO_{\P^1}(-1)} [2]
\end{equation}
As a check, it is straightforward to verify that this holds on $\cO_{\P^1}(-1)$, see Remark~\ref{rem.redconifold}. In this example, \eqref{eq.3fold} was proved by the author and \mbox{E.~Segal} by a very different method using {grade restriction windows}~\cite[Section~2.3]{DS1}.
\end{eg}

I outline the proof of Theorem~\ref{keythm.blowupB}. Assumption~(\ref{cor.blowupBgeomAssmA}) gives a family of \mbox{$3$-fold} ordinary double point singularities parametrized by $\blowupLocusB$ (see Remark~\ref{rem.lowdimegs} below for a construction). I~analyze such geometries in Section~\ref{sec.eg}, obtaining Theorem~\ref{cor.blowupBgeom}. This establishes part~(\ref{keythm.blowupBA}) above, and also shows that the following is satisfied.

\begin{keyassm}[\optbracket{Assumption~\ref{assm.res}}]\label{assm.div} The restriction of $ \cO_{\blowupProjA}(1)$ to $\excBlowupB$ is
$\cO_{\blowupProjB}(1) \otimes \blowupProjB^* \divcasetwist $
for some invertible sheaf~$ \divcasetwist$.\footnote{As a mnemonic, we have `d' for degree, noting that this assumption gives that the invertible sheaf $ \cO_{\blowupProjA}(1)$ has degree~$1$ on each fibre of $\blowupProjB$.}  \end{keyassm}

\noindent The rest of the proof is then completed in Section~\ref{section.divisor}. Using part~(\ref{keythm.blowupBA}) we have a semiorthogonal decomposition as follows, using that $Z$ is regularly embedded in $\blowupLocusA$ by assumption~(\ref{cor.blowupBgeomAssmB}). 
\begin{equation}\label{eqn.keysodB}
\begin{aligned}
\D(\excBlowupSingular) & = \big\langle \blowupProjA^* \D(\blowupLocusA),  \D(\blowupLocusB) \big\rangle 
\end{aligned}
\end{equation}
Writing now $\incResInBlowupAmbient\colon\excBlowupSingular \into \resolution$, the pushforward functor $\incResInBlowupAmbient_*\colon\D(\excBlowupSingular)\to\D(\resolution)$ is spherical, again using that $\incResInBlowupAmbient$ is the embedding of a Cartier divisor. 

I find that Halpern-Leistner--Shipman's criterion holds for~\eqref{eqn.keysodB} using Assumption~\ref{assm.div} above, see Proposition~\ref{prop.blowupB}. Restricting $\incResInBlowupAmbient_*$ to the components of~\eqref{eqn.keysodB} gives respectively $\gun$ and $\hun$, up to tensoring by an invertible sheaf. It follows that these functors are spherical, yielding part~(\ref{keythm.blowupBB}). I furthermore establish a factorization of the twist $\tw[\incResInBlowupAmbient_*] \cong \placeholder\otimes \cO(\excBlowupSingular)$ with factors $ \tw[\gun]$ and $\tw[\hun]$. Along with Theorem~\ref{thm.mainintro}, I thence obtain two independent descriptions of $ \tw[\gun]$. Comparing these descriptions, and relating the invertible sheaves which appear in them, finally gives the isomorphism~(\ref{keythm.blowupBC}).

\def\threefoldpicturefigure{
\threefoldpicture{0}
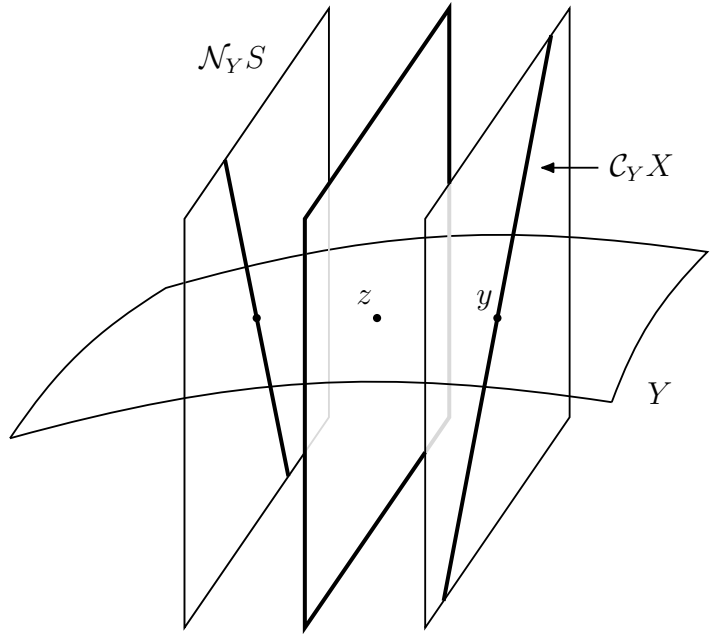
\captionsized{ Sketch of the singular geometry of $\singular$ along $\blowupLocusA$ in the setting of Theorem~\ref{keythm.blowupB}. The vertical planes are fibres of~$\normalof{\blowupLocusA}{\ambient}$. This bundle contains the normal cone $\normalcone{\blowupLocusA}{\singular}$, shown by thickened lines. Recall that $\blowupLocusB \subset \blowupLocusA$ is the codimension~$2$ locus where $\sec$~vanishes: over $z$ in $\blowupLocusB$, $\normalcone{\blowupLocusA}{\singular}$ coincides with~$\normalof{\blowupLocusA}{\ambient}$; over $y$ not in $\blowupLocusB$, $\normalcone{\blowupLocusA}{\singular}$ is cut out of~$\normalof{\blowupLocusA}{\ambient}$ by a non-zero linear function~$\sec(y)$.}  
\label{fig.singn1}
}

\begin{figure}[htb]
\opt{ams}{\threefoldpicturefigure}\opt{comp}{\begin{center}\threefoldpicturefigure\end{center}}
\end{figure}

In the setting of Theorem~\ref{keythm.blowupB}, I therefore find three different spherical functors to~$\D(\resolution)$, namely $\compMap^*$, $\gun$, and~$\hun$, whose twists coincide up to inverses, homological shifts, and tensoring by an invertible sheaf. In Section~\ref{sec.objects} I\opt{ams}{ }\opt{comp}{~}give an explicit description of their action on objects.

\begin{rem}\label{rem.lowdimegs} Theorem~\ref{keythm.blowupB} encompasses many examples in all dimensions, see Section~\ref{section.Examples}. I~give one construction here. Take a smooth scheme $\blowupLocusA$ with a locally free sheaf $\bunsec$ of rank~$2$ and a regular section $\sec$ cutting out $\blowupLocusB$. Let $\ambient$ be the total space $\Tot \bunsec^\vee$ and embed $\blowupLocusA$ as its zero section. Then $\sec$ tautologically induces a global function on $\ambient$, see Assumption~\ref{assm.cutnormal}. Taking $\singular$ to be its zeroes, the assumptions of  Theorem~\ref{keythm.blowupB} are satisfied by Proposition~\ref{prop.localThmC}. In particular, $\resolution$ is a small (thence crepant) resolution of $\singular$.
\end{rem}

I next study contractions which are not necessarily small.

\subsection{Blowups in general codimension} I describe $ \tw[\compMap^*] $ for general $n\coloneqq\codim_\singular \blowupLocusA\geq 1$, using additional global assumptions on~$\singular$ and the ambient space~$\ambient$. I find a higher degree relation in the autoequivalence group of $\D(\resolution)$, given in~(\ref{thm.mainintroBB}) below.

\begin{keythm}[\optbracket{Theorem~\ref{thm.mainB}}]\label{thm.mainintroB} In the setting of Section~\ref{sect.setting} suppose  that:
\begin{enumerate}[label={(\roman*)},ref={\roman*}]
\item \label{thm.mainintroBassmA} $\ambient$ is projective.
\item \label{thm.mainintroBassmB} $\singular$ is an anticanonical divisor in $\ambient$. \end{enumerate}
We may put 
\begin{equation*}
	\begin{tikzpicture}[xscale=\compdiagscale]
		\node (A) at (-0.25,0) {$ \gun_\index\colon\D(\blowupLocusA)$};
		\node (B) at (1,0) {$\D(\excBlowupSingular)$};
		\node (C) at (3,0) {$\D(\excBlowupSingular)$};
		\node (D) at (4,0) {$\D(\resolution)$};

		\draw [->] (A) to node[above]{\compstyle $\blowupProjA^*$} (B);
		\draw [->] (B) to node[above]{\compstyle $\otimes \cO_{\blowupProjA}(\index-1)$} (C);
		\draw [->] (C) to (D);
	\end{tikzpicture}
\end{equation*}
for $\index\in\mathbb{Z}$, where the last functor is  pushforward. Then:

\begin{enumerate}[label={(\arabic*)},ref={\arabic*}]
\item\label{thm.mainintroBA} $\gun_\index$ is spherical for each $\index\in\mathbb{Z}$.
\item\label{thm.mainintroBB} There is  an isomorphism
\begin{equation*}
 \tw[\compMap^*] (\placeholder)\otimes \normal{\resolution} \cong \big(\tw[\gun_1] \comp \dots \comp \tw[\gun_n]\big)^{-1}(\placeholder)[2]
\end{equation*}
between autoequivalences of $\D(\resolution)$.\end{enumerate}
Here $\normal{\resolution}$ is the invertible normal sheaf of $\resolution$ in $\resolutionAmbient=\blowupAmbient$. 
\end{keythm}

\begin{rem} 

Parts (\ref{thm.mainintroBassmA}) and (\ref{thm.mainintroBassmB}) imply that $\singular$ is Calabi--Yau, and therefore supplement our assumption that $\resMap$ is crepant (which may be thought of as a {relative} Calabi--Yau assumption). They could surely be relaxed, in particular by working relative to the base, see Remark~\ref{rem.relgen} for discussion.
\end{rem}

Theorem~\ref{thm.mainintroB} is proved by again taking the decomposition~\eqref{eqn.keysod} for the blowup $\resolutionAmbient=\blowupAmbient$. This has $n+1 \geq 2$ components. To describe $ \tw[\compMap^*] $ we would like to factor the twist $\tw[\incResInBlowupAmbient^*] $ as previously in Theorem~\ref{thm.mainintro}, but this time into $n+1$~factors. A factorization criterion for this situation is given by Addington and Aspinwall~\cite{AddingtonAspinwall}, credited to Kuznetsov, see Theorem~\ref {thm.sodsphserre}. This assumes a Serre functor for the source category $\D(\resolutionAmbient)$, which exists using~(\ref {thm.mainintroBassmA}).  The criterion may then be checked using~(\ref {thm.mainintroBassmB}).
 
 The rest of the argument is then similar to Theorem~\ref{thm.mainintro}. In particular, spherical functors $\gun_1, \dots, \gun_n$ are obtained by restriction of $\incResInBlowupAmbient^*$ to the last $n$~components of the decomposition~\eqref{eqn.keysod}. I thereby find a factorization of the twist $\tw[\incResInBlowupAmbient^*] \cong\placeholder\otimes  \normal{\resolution}^\vee[2]$ which rearranges to give  the isomorphism~(\ref{thm.mainintroBB}).

\stdskip

For a first example, consider the following.

\begin{eg}\label{eg.quarticK3} Let $\singular$ be a quartic K3 surface in $\ambient = \mathbb{P}^3$ with a node~$x$, and put $\blowupLocusA = \{x\}$ so that $\resolution \to \singular$ resolves this node. Then $\excBlowupSingular\cong\mathbb{P}^1$, and (\ref{thm.mainintroBB})~describes $\tw[\compMap^*]$ in terms of twists  by spherical objects on $\resolution$ as follows.
\begin{equation*}\label{eq.surf}
 \tw[\compMap^*] (\placeholder)\otimes \normal{\resolution} \cong \big(\tw[\cO_{\mathbb{P}^1}] \comp \tw[\cO_{\mathbb{P}^1}(2)]\big)^{-1}(\placeholder)[2]
\end{equation*}
The geometry is illustrated in Figure~\ref{fig.surf}. For details, see Example~\ref{eg.quarticK3ext}.
\end{eg}

\begin{figure}[htb]
\begin{center}
\surfpicture
\captionsized{Local sketch of quartic K3 surface $\singular$ from Example~\ref{eg.quarticK3} with node $x$, in smooth ambient $3$-fold $\ambient= \mathbb{P}^3$. Blowing up~$\singular$ at $\blowupLocusA=\{x\}$ gives an exceptional $\excBlowupSingular \cong \P^1$ which is conic in~$\excBlowupAmbient\cong \P^2$.}
\label{fig.surf}
\end{center}
\end{figure}

\subsection{Compatibility with base change}

I explain a useful compatibility of the twist functors~$\tw[\compMap^*]$ with base change, in particular to relate twists in different dimensions.
  
Consider then the morphism $\compMap \colon \resolution \to \ambient$ as above in Section~\ref{sect.setting}, and take $\compMap' \colon \resolution' \to \ambient'$ its base change along some $\ambient' \to \ambient$. Suppose that $\compMap'$ may also be obtained as in Section~\ref{sect.setting}. Then I give an intertwinement of the corresponding twists as follows, see Proposition~\ref{prop.twistbc}, via the morphism~$\ambientSourcebc\colon \resolution' \to \resolution$ induced by the base change.
\beq
\tw[\compMap^*] \ambientSourcebc_* \cong \ambientSourcebc_* \tw[\compMap'^*]
\eeq
In particular, we may take an embedding $\ambient' \into \ambient$, for example as follows. 

\begin{eg} Let $\singular$ be a {one-parameter deformation} of a quartic K3 surface $\singular'$ with node~$x$, embedded in a smooth one-parameter deformation $\ambient$ of $\ambient'=\mathbb{P}^3$. Require that~$\singular$ has a 3-fold ordinary double point at~$x$. Taking appropriate resolutions, we are in the setting above, and comparing Theorems~\ref{keythm.blowupB} and~\ref{thm.mainintroB} yields a relation in the autoequivalence group of $\D(\resolution')$ which follows, after some work, from known facts: for details, see Example~\ref{eg.defquarticK3}, with the relation given in Proposition~\ref{prop.K3rel}.
\end{eg}

\subsection{Canonicity}

It is interesting to explore to what extent a derived symmetry might be {canonically} associated to a contraction. Inspired by the above results, I make the following. 

\begin{defn}\label{defn.hypertw} Assuming that the twist $\tw[\compMap^*]$ is an autoequivalence, and using Assumption~\ref{assm.embed}, define a \emph{hypersurface twist} as follows.
\beq
\hypertw = \tw[\compMap^*](\placeholder) \otimes \resMap^* \normal{\singular} [-2]
\eeq
\end{defn}
A priori, this  depends on a choice of embedding of the hypersurface $\singular$ in the ambient space~$\ambient$.
However, the  above results suggest that, under various assumptions, it may be possible to  describe $\hypertw$ in terms of the geometry of the contraction $\resMap$.
\begin{itemize}
\item In the settings of Theorem~\ref{thm.mainintro} and~\ref{thm.mainintroB}, we have
\beq 
\hypertw \cong \tw[\gun]^{-1}(\placeholder) \otimes \cO(\excBlowupSingular)
\eeq
and 
\beq 
\hypertw \cong \big(\tw[\gun_1] \comp \dots \comp \tw[\gun_n]\big)^{-1}(\placeholder) \otimes \cO(n\excBlowupSingular)
\eeq
respectively, using Proposition~\ref{lem.normals}.
\item In the setting of Theorem~\ref{keythm.blowupB}, we immediately have $\hypertw \cong \tw[\hun] $.
\end{itemize}

I was not able to recover the right-hand sides of the above equations from the morphism~$\resMap$ alone, but I give the following partial results.
\begin{itemize}
\item For \mbox{$n\geq 2$,} the locus $\blowupLocusA$ may be characterized by the property that its points have positive-dimensional preimage under $\resMap$, using 
Proposition~\ref{prop.smCar2}(\ref{prop.smCarR}). 
\item For $n=1$ in the setting of Theorem~\ref{keythm.blowupB}, the locus $\blowupLocusB$ may be characterized scheme-theoretically in terms of the singularities of $\singular$, see Proposition~\ref{prop.singloc}.
\end{itemize}
Note also  that for a flopping contraction $\resMap$ where equivalences
\beq\funFlopFlop\colon \D(\resolution) \to\D(\resolution') \to \D(\resolution)
\eeq
are known, then $\funFlopFlop$ may be taken as a derived symmetry canonically associated to the contraction. Remark~\ref{rem.BB} explains why we may hope that $\hypertw\cong\funFlopFlop^{-1}$ in a setting where they are both defined, but I do not pursue this claim further here.

\subsection{Further questions}

Given Theorem~\ref{thm.pullbacksph} it is natural to ask for which morphisms $\psi$ is the functor $\psi^*$ spherical. I am grateful to Evgeny Shinder for this interesting question. In particular, my proofs crucially use that Cartier divisors yield spherical functors: the same is true of ramified double covers~\cite[Lemma 2.9]{KP}, so these may likewise give spherical functors in combination with crepant contractions.

It is also natural to ask if there is a mirror statement to Theorem~\ref{thm.pullbacksph}, relating to known spherical functors in symplectic geometry. See for instance~\cite{KPS} and~\cite[Introduction]{GJ} for discussion of the mirror to the spherical functor $\incResInBlowupAmbient_*$ for an anticanonical embedding~$\incResInBlowupAmbient$.

It would be interesting to study whether Assumption~\ref{assm.blowup} can be relaxed to allow singular~$\blowupLocusA$. For instance, in general a flopping contraction~$f$ of a \mbox{3-fold}
 is a blowup along singular~$\blowupLocusA$. In this setting the author and Wemyss constructed a derived symmetry $\nctw$ of the 3-fold by using noncommutative deformation techniques~\cite{DW1,DW3}, so it would be good to investigate whether the constructions in this paper may extend to recover $\nctw$. Note that they already recover \emph{commutative} deformation algebras, see Example~\ref{eg.pagoda}.

Finally, Kuznetsov and Shinder~\cite{KS} develop a theory of categorical absorption of singularities. They apply it to varieties with isolated ordinary double points, and study its relationship with smoothings. It would be interesting to connect it with this work.

\begin{acks} \acksize 
I am grateful to A.~
Bodzenta for our collaboration~\cite{BD} which greatly aided this project, and for discussions on relating the results here for~$n=1$ to periodic semiorthogonal decompositions; to~E.~
Segal for enjoyable attempts to remember details of our work~\cite{DS1}, and an important conversation on an early version of Theorem~\ref{keythm.blowupB} where he noted that $\normalof{\blowupLocusA}{\ambient}$ need not be split; to M.\ Wemyss for our collaboration since~\cite{DW1}, and for all we have learnt together, not least about typesetting;
to Yu Zhao for inspiring discussions and correspondence on an approach to Theorem~\ref{keythm.blowupB} using derived algebraic geometry, and for sharing~\cite{Zha1,Zha2}. I~thank the author of a quick opinion for helpful comments and for sharing a simplified argument for Proposition~\ref{lem.ideals}, improving on a previous version.

I thank P.~
Achinger, N.~
Addington, A.~
Bondal, C.~
Fietz, Y.~
Ito, A.~
Keating, C.~N.~
Leung, B.~
Pauwels, M.~
Romo, I.~
Smith, and Y.~
Toda 
for helpful conversations on and around this project. I gratefully acknowledge the hospitality of ANU in Canberra, Kavli IPMU, U.~Kyoto, Makerere~U.\ in Kampala, and U.~Oregon, as well as warm welcomes and financial support from IMS at CUHK, J.~
Knapp and Gufang~
Zhao at U.~Melbourne, and SMRI in Sydney, during extended visits.
\end{acks}

\subsection*{Conventions} 
{ 
Schemes are over an algebraically closed field of characteristic~0, and are taken to be of finite type, with morphisms of finite type. They are furthermore taken to be reduced and separated, unless explicitly stated.\footnote{In particular, I allow non-reduced $\blowupLocusB$ in Theorem~\ref{keythm.blowupB}, and a non-reduced exceptional locus~$\excBlowupB$.}  I say a scheme is {Calabi--Yau} if it has at worst Gorenstein singularities and trivial canonical sheaf.

A {contraction} $\resMap\colon\resolution \to\singular$ is a projective morphism with $ \RDerived^0\resMap_*\cO_\resolution \cong \cO_\singular$, and a morphism~$\resMap$ is birational if it induces a bijection of sets of generic points of irreducible components, which in turn induces isomorphisms of local rings at those points.\footnote{I allow non-normal and reducible schemes, see Remark~\ref{eg.relaxed} for an example. For details of this definition of birational, see for instance~\cite[\href{https://stacks.math.columbia.edu/tag/01RO}{Definition 01RO}]{stacks-project}.} 

I write $\normalcone{B}{A}$ for the normal cone of~$B$ in~$A$, and reserve the notation $\normalof{B}{A}$ for the case when this is a bundle, in particular when $B$ is an effective Cartier divisor in~$A$. I sometimes write simply $\normal{B}$ when the embedding in $A$ can be clear from context.}

The bounded derived category of coherent sheaves is denoted by $\D(X)$, and functors are taken to be derived. Components of semiorthogonal decompositions will be required to be admissible (rather than only left or right admissible).

\section{Preliminaries}
\label{sec.sphfun}

\subsection{Spherical functors}
Take a functor $\fun$ of enhanced triangulated categories with a right adjoint $\Radj\fun$. Suppose  that we have triangles of Fourier--Mukai functors associated to the adjunction counit and unit, as follows.
\begin{equation*}
\fun \comp \Radj\fun \to \id \to \tw[\fun] \to 
\qquad
\ctw[\fun] \to \id \to \Radj\fun \comp \fun \to  \end{equation*}
We call $\tw[\fun]$ and $\ctw[\fun]$ the \emph{twist} and \emph{cotwist} of $\fun$, respectively.\footnote{Different conventions are used, for instance Addington~\cite{Addington} writes $\ctw[\fun]$ where we would write $\ctw[\fun][1]$. That convention has the advantage that the cotwist arises from a cone, rather than a shifted cone. However the statement of Proposition~\ref{prop.divsph} becomes less symmetric and memorable.}
If furthermore~$\fun$ has a left adjoint $\Ladj\fun$ then we have the following.

\begin{defn} We say that $\fun$~is \emph{spherical} if both $\tw[\fun]$ and $\ctw[\fun]$ are  autoequivalences. 
\end{defn}

This is one of several equivalent characterisations of spherical, see Anno and Logvinenko~\cite{Anno,AL} and also Kuznetsov~\cite[Section~2.5]{Kuz}.

\subsection{Spherical objects}\label{sec.sphobj}
For the special case of a spherical functor
$\fun\colon \D(\pt) \to \D(X)$
we say that $\cE = \fun(\cO_{\pt}) \in \D(X)$ is a \emph{spherical object}, following the sense of Seidel--Thomas~\cite{ST}. In this case we abuse notation by writing $\twObj{X}{\cE}$ for $\tw[\fun]$ so that we have the following.
\begin{equation*}
\cE \otimes \comp \Hom_X(\cE,\placeholder) \to \id \to \twObj{X}{\cE} \to
\end{equation*}
We refer to $\twObj{X}{\cE}$ as the twist by the object $\cE$.

\subsection{Cartier divisors} These give fundamental examples of spherical functors.

\begin{prop}[\cite{Anno,Addington,Kuz}]\label{prop.divsph} Take $\inc\colon \gendivisor \into \gendivisorambient$ the inclusion of a  Cartier divisor.

\begin{enumerate}
\item\label{prop.divsphA} The functor $\inc^*\colon \D(\gendivisorambient) \to \D(\gendivisor)$ is spherical with:
\begin{equation*}
\tw[\inc^*] \cong \placeholder\otimes  \normal{\gendivisor}^\vee[2]
\qquad
\ctw[\inc^*] \cong \placeholder\otimes \ideal_\gendivisor
\end{equation*}

\item\label{prop.divsphB} The functor $\inc_*\colon \D(\gendivisor) \to \D(\gendivisorambient)$ is spherical with:
\begin{equation*}
\tw[\inc_*] \cong \placeholder\otimes \ideal_\gendivisor^\vee
\qquad
\ctw[\inc_*] \cong \placeholder\otimes \normal{\gendivisor}[-2]
\end{equation*}
\end{enumerate}
\end{prop}

This appeared, assuming smoothness, in \cite[Section~2.2, (4) and~(5)]{Addington}.  It is proved without smoothness in \cite[Example~3.1, Proposition~3.4]{Kuz}, and stated in \cite[Lemma~2.8]{KP}.

\begin{rem}\label{rem.adjs} The adjoints required for Proposition~\ref{prop.divsph} are
$\inc_! \dashv \inc^*  \dashv \inc_* \dashv \inc^!$
where:
\begin{equation*}
\inc_! = \inc_* (\placeholder \otimes \normal{\gendivisor}[-1])
\qquad
\inc^! =  \normal{\gendivisor}[-1] \otimes \inc^* (\placeholder) 
\end{equation*}
Note that $\inc_*$ preserves the bounded coherent derived category because $\inc$ is proper, and similarly for $\inc^*$ because, although $\gendivisorambient$ (and furthermore~$\gendivisor$) is not necessarily smooth, $\inc$~is the inclusion of a Cartier divisor.
\end{rem}

\begin{rem}\label{rem.inverses} Parts (\ref{prop.divsphA}) and (\ref{prop.divsphB}) above are related by taking inverses, as follows. Given triangles of Fourier--Mukai functors associated to the adjunction $\Ladj\fun \dashv \fun$
\begin{equation*}
\twd[\fun] \to \id \to   \fun\comp\Ladj\fun \to
\qquad
\Ladj\fun \comp \fun \to \id \to \ctwd[\fun] \to
\end{equation*}
then $\twd[\fun]$ and $\ctwd[\fun]$ are, by construction, $\ctw[\Ladj\fun]$ and $\tw[\Ladj\fun]$ respectively. For spherical~$\fun$, they furthermore yield $\tw[\fun]^{-1}$ and $\ctw[\fun]^{-1}$ respectively. 
\end{rem}

\begin{rem} For a variant of Proposition~\ref{prop.divsph} in a triangulated setting using square root stacks and periodic semiorthogonal decompositions, see recent work of the author and Bodzenta~\cite[Corollary~1.5]{BD}. 
\end{rem}

\subsection{Decompositions} We will extensively use semiorthogonal decompositions of triangulated categories. For a survey see~\cite{KuzICM}. Recall that a full triangulated subcategory~$\mathsf{A}$ of a triangulated category $\mathsf{D}$ is \emph{left} (respectively \emph{right}) \emph{admissible} if its embedding has a left (respectively right) adjoint. We say that it is \emph{admissible} if its embedding has both adjoints.

Now write ${}^\perp\mathsf{A}$ for the \emph{left orthogonal} to $\mathsf{A}$, the full (triangulated) subcategory of~$\mathsf{D}$ with objects having only zero morphisms to $\mathsf{A}$. Assuming $\mathsf{A}$ and ${}^\perp\mathsf{A}$ are admissible, we obtain a semiorthogonal decomposition as follows.\footnote{Often `semiorthogonal decomposition' describes the weaker notion where the two components are only required to be left and right admissible, respectively.}
\beq\mathsf{D} = \big\langle \mathsf{A}, {}^\perp\mathsf{A} \big\rangle\eeq

\subsection{Factoring twists} 
\label{sect.factoring}

The following result of Halpern-Leistner--Shipman is key for our proofs of Theorems~\ref{thm.pullbacksph}, \ref{thm.mainintro} and \ref{keythm.blowupB}. It allows us to factor the twist of a spherical functor when its source category has a semiorthogonal decomposition satisfying an appropriate compatibility. It will be applied throughout to decompositions associated to blowups.

\begin{thm}[{\cite[Section~4.3]{HLShi}}]\label{thm.sodsph} Let $\fun\colon \mathsf{D} \to \mathsf{D}'$ be a spherical functor such that:
\begin{equation*}
\mathsf{D} = \big\langle \mathsf{A}, \mathsf{B} \big\rangle = \big\langle \ctw[\fun]\mathsf{B}, \mathsf{A} \big\rangle
\end{equation*}
Then $\fun_{\mathsf{A}}$ and $\fun_{\mathsf{B}}$, the restrictions of $\fun$ to $\mathsf{A}$ and $\mathsf{B}$ respectively, are spherical, and: 
\begin{equation*}
\tw[\fun] \cong \tw[\fun_{\mathsf{A}}] \comp \tw[\fun_{\mathsf{B}}]
\end{equation*}
\end{thm}

\section{Pullback from the ambient}\label{sect.pullback}

In this section I prove Theorem~\ref{thm.pullbacksph}, that the derived pullback functor~$\compMap^*$ from the ambient space~$\ambient$ is spherical, by studying semiorthogonal decompositions associated to a blowup of~$\ambient$. Continuing the line of argument yields  Theorem~\ref{thm.mainintro}. Along the way I give some useful consequences of crepancy (Propositions~\ref{lem.normals} and~\ref{lem.ideals}) which are also used later (in particular in the proofs of Theorem~\ref{keythm.blowupB} and~\ref{thm.mainintroB} respectively).

\opt{ams}{\stdskip}
Throughout take the setting of Section~\ref{sect.setting} with $\blowupLocusA \hookrightarrow \singular \hookrightarrow \ambient$. Then $\resMap\colon \resolution \to \singular$ is the blowup of $\singular$ along~$\blowupLocusA$ by~Assumption~\ref{assm.blowup}. Letting $\blowupAmbientMap\colon \resolutionAmbient \to \ambient$ be the blowup of $\ambient$ along this same\opt{ams}{ }\opt{comp}{~}$\blowupLocusA$, we have a commutative diagram as follows.
\begin{equation}
\begin{aligned}\label{eq.resandamb}
    \begin{tikzpicture}[scale=\stdscale]
    	\node (singular) at (0,0) {$\singular$};
	\node (resolution) at (0,1) {$\resolution$}; 
	\node (ambient) at (1,0) {$\ambient$};
	\node (blowupAmbient) at (1,1) {$\resolutionAmbient$};
	\draw [right hook->] (singular) to node[below]{\stdstyle $\incInAmbient$} (ambient);
	\draw [->] (resolution) to node[left]{\stdstyle $\resMap$} (singular);
	\draw [right hook->] (resolution) to node[above]{\stdstyle $\incResInBlowupAmbient$} (blowupAmbient);
	\draw [->] (blowupAmbient) to node[right]{\stdstyle $\blowupAmbientMap$} (ambient);
	\draw [->] (resolution) to  node[above right]{\stdstyle $\compMap$} (ambient);
	\end{tikzpicture}
\end{aligned}
\end{equation}

Assumptions~\ref{assm.embed} and~\ref{assm.blowup} control the geometry of $\blowupAmbientMap$ as follows.

\begin{prop}\label{prop.smCar} Recall $\blowupLocusA$ is smooth of constant codimension~$n$ in $\singular$ by Assumption~\ref{assm.blowup}.

\begin{enumerate}
\item\label{prop.smCarZ} $n\geq 1$.

\item\label{prop.smCarQ} $\blowupLocusA$ is regularly embedded in $\ambient$ with codimension $n+1$.

\item\label{prop.smCarY} $\resolutionAmbient$  and $\excBlowupAmbient$ are smooth equidimensional.

\item\label{prop.smCarX} The relative canonical sheaf $\omega_\blowupAmbientMap = \omega_{\resolutionAmbient} \otimes \blowupAmbientMap^* \omega_\ambient^\vee$ is isomorphic to $\cO(n\excBlowupAmbient)$.
\end{enumerate}
\end{prop} 

\begin{proof} 
For~(\ref{prop.smCarZ}), suppose for a contradiction that $n=0$. Note that each connected component~$\blowupLocusA'$ of~$\blowupLocusA$ is irreducible by smoothness. Then $\blowupLocusA'$ would coincide with some irreducible component $\singular'$ of~$\singular$, and so the blowup~$\resolution$ would have no irreducible component corresponding to $\singular'$, contradicting our assumption that $\resMap$ is birational.\footnote{Recall that, according to our conventions, such a morphism induces a bijection between irreducible components by definition.}
For~(\ref{prop.smCarQ}), we use that, by Assumption~\ref{assm.embed}, $\singular$ is a hypersurface in~$\ambient$ smooth equidimensional. For~(\ref{prop.smCarY}), smoothness of $\resolutionAmbient$ may then be checked locally using smoothness of~$\ambient$, and equidimensionality  of $\resolutionAmbient$ follows from equidimensionality of~$\ambient$. Then (\ref{prop.smCarX}) is standard.
\end{proof}

Assumptions~\ref{assm.embed} and~\ref{assm.blowup} furthermore restrict the geometry of $\resMap\colon\resolution\to\singular$.

\begin{prop}\label{prop.smCar2} We have the following:

\begin{enumerate}
\item\label{prop.smCarA} $\singular$ and $\resolution$ are effective Cartier divisors in $\ambient$ and $\resolutionAmbient$ respectively.

\item\label{prop.smCarB} $\singular$ and $\resolution$ are Gorenstein.

\item\label{prop.smCarR} For $y\in\blowupLocusA$, the fibre $ \resMap^{-1}(y)$ has pure dimension $n-1$ or $n$.

\end{enumerate}
\end{prop} 
\begin{proof}
\firstproofpart{\opt{ams}{1}\opt{comp}{i}} For this note that $\singular$ and $\resolution$ are hypersurfaces in smooth $\ambient$ and $\resolutionAmbient$ respectively, using Assumption~\ref{assm.embed} and Proposition~\ref{prop.smCar}(\ref{prop.smCarY}), and are  reduced by assumption.

\proofpart{\opt{ams}{2}\opt{comp}{ii}} This follows as $\singular$ and $\resolution$ are effective Cartier divisors in  smooth schemes.

\proofpart{\opt{ams}{3}\opt{comp}{iii}} Using Proposition~\ref{prop.smCar}(\ref{prop.smCarQ}), for closed points $y\in\blowupLocusA$ we have $ \blowupAmbientMap^{-1}(y)\cong \P^n$. Now $\resMap^{-1}(y)=\resolution\cap\blowupAmbientMap^{-1}(y)$ by commutativity of~\eqref{eq.resandamb}. This is non-empty by surjectivity of $\resMap$, and its irreducible components are codimension~$0$ or~$1$ in~$\P^n$ by (\ref{prop.smCarA}). In the former case $ \resMap^{-1}(y) \cong \P^n$. Otherwise 
$\resMap^{-1}(y)$ has constant codimension~$1$ in $\P^n$, hence the claim.
\end{proof}

The following is a consequence of Assumption~\ref{assm.crep}.

\begin{prop}\label{lem.normals} The invertible normal sheaves of $\singular$ and $\resolution$ are related by: 
\begin{enumerate}
\item\label{lem.normals1} $\resMap^* \normal{\singular} \cong \normal{\resolution} \otimes \incResInBlowupAmbient^*\omega_\blowupAmbientMap $
\item\label{lem.normals2} $\incResInBlowupAmbient^*\omega_\blowupAmbientMap \cong \cO(n\excBlowupSingular) $
\end{enumerate}

\end{prop}
\begin{proof} 
\firstproofpart{\opt{ams}{1}\opt{comp}{i}} The commutative square~\eqref{eq.resandamb} gives $ \omega_\resMap \otimes \resMap^* \omega_\incInAmbient \cong \omega_\compMap \cong \omega_\incResInBlowupAmbient \otimes \incResInBlowupAmbient^*\omega_\blowupAmbientMap$.
But $\omega_\resMap$ is trivial  by Assumption~\ref{assm.crep}, and $\omega_\incInAmbient = \normal{\singular}$ and $\omega_\incResInBlowupAmbient = \normal{\resolution}$ by Proposition~\ref{prop.smCar2}(\ref{prop.smCarA}).

\proofpart{\opt{ams}{2}\opt{comp}{ii}} Note that $\incResInBlowupAmbient^*\cO(\excBlowupAmbient) \cong\cO(\excBlowupSingular)$ and use Proposition~\ref{prop.smCar}(\ref{prop.smCarX}). 
\end{proof}

The following compares the Cartier divisors $\resolution$ and $\singular$ via the blowup~$\blowupAmbientMap$.

\begin{prop}\label{lem.ideals} We have the following isomorphisms on $\resolutionAmbient$.
\begin{enumerate}
\item\label{lem.idealsX} $\ideal_\resolution \cong \blowupAmbientMap^!\ideal_\singular $
\item\label{lem.idealsB} $\ideal_\resolution \cong \blowupAmbientMap^*\ideal_\singular \otimes \cO(n\excBlowupAmbient)$
\end{enumerate}

\end{prop}
\begin{proof} Note first that $\dim\blowupAmbientMap=0$, so~(\ref{lem.idealsX}) and~(\ref{lem.idealsB}) are equivalent by construction of $\blowupAmbientMap^!$ and Proposition~\ref{prop.smCar}(\ref{prop.smCarX}). By Assumption~\ref{assm.blowup}, $\blowupLocusA$ is smooth. Assuming for now that $\blowupLocusA$ is furthermore connected, we have
\begin{equation*}\label{eqn.normalequiv} \blowupAmbientMap^* \singular = \resolution + m\excBlowupAmbient\end{equation*}
in~$\operatorname{Pic}(\resolution)$, where $m$ is the multiplicity along $\blowupLocusA$ of $\singular$ as a subscheme of $\ambient$. To show~(\ref{lem.idealsB}), I~claim then that $m=n$. 

Proposition~\ref{prop.smCar}(\ref{prop.smCarX}) yields
$K_\resolutionAmbient + \resolution 
  = \blowupAmbientMap^*( K_\ambient +  \singular)  + (n-  m)\excBlowupAmbient$
and so, restricting to $\resolution$,
\begin{equation*} K_{\resMap} = K_\resolution - \resMap^* K_\singular = (n-  m)\excBlowupSingular\end{equation*}
using adjunction and commutativity of \eqref{eq.resandamb}. But $K_{\resMap} = 0$ by  Assumption~\ref{assm.crep}, so 
\begin{equation}\label{eq.multExc} 
(n-  m)\excBlowupSingular = 0
\end{equation} 
 in~$\operatorname{Pic}(\resolution)$. Note that $-k\excBlowupSingular$ is very ample relative to~$\resMap$ for all $k\in\mathbb{N}$. 

Suppose for a contradiction that $m\neq n$, so that $\cO_\resolution$ is very ample relative to~$\resMap$ by \eqref{eq.multExc}.  But $\resMap$ is a contraction, hence a surjection with $\RDerived^0\resMap_* \cO_\resolution \cong \cO_\singular$, so then $\resMap$ embeds $\resolution$ into $\operatorname{Proj}_\singular\cO_\singular=\singular$ over the base~$\singular$, making $\resMap$ an isomorphism. Using Assumption~\ref{assm.blowup} and the universal property of blowup, $\blowupLocusA$ is then a Cartier divisor in~$\singular$, giving $n=1$. But then a fibre of $\resMap$ over $\blowupLocusA$ is a degree~$m$ hypersurface in $\mathbb{P}^1$ hence $m=1$, a contradiction.

Finally, if $Y$ is not connected, we may check~(\ref{lem.idealsB}) after base change to a neighbourhood of each component of~$\blowupLocusA$ in~$\ambient$ which is disjoint from the other components of~$\blowupLocusA$.
\end{proof}

\begin{rem} The above proof of Proposition~\ref{lem.ideals} is the only place where we use our assumption that $\resMap$ is a contraction. Therefore we may replace this assumption with the requirement that the multiplicity $m$ along $\blowupLocusA$ of $\singular$ as a subscheme of $\ambient$ is constant, and equal to $n \coloneqq \codim_\blowupLocusA\singular $.
\end{rem}

\begin{eg}\label{eg.cone} Let $\singular$ be the affine cone over a reduced  hypersurface $H$ of degree $n$ in $\mathbb{P}^n$, with $\ambient=\field^{n+1}$ its ambient space, for $n\geq 2$. Take $\resMap$ to be the blowup of $\singular$ at the vertex~$\{0\}$ of the affine cone, so that  $\blowupLocusA=\{0\}$. Then the blowups $\resolution$ and $\resolutionAmbient$ and their exceptional loci form a fibre square as follows.
\begin{equation*}
    \begin{tikzpicture}[scale=\stdshrink]
    	\node (excBlowupAmbient) at (1,0) {$\mathbb{P}^n$};
	\node (excBlowupSingular) at (0,0) {$H$};
	\node (resolutionAmbient) at (1,-1) {$\resolutionAmbient$};
	\node (resolution) at (0,-1) {$\resolution$}; 
	\draw [left hook->] (excBlowupSingular) to (resolution);
	\draw [right hook->] (resolution) to node[above]{\stdstyle $\incResInBlowupAmbient$} (resolutionAmbient);
	\draw [left hook->] (excBlowupAmbient) to (resolutionAmbient);
	\draw [right hook->] (excBlowupSingular) to (excBlowupAmbient);
	\end{tikzpicture}
\end{equation*} 

Consider $\resolutionAmbient$ as the total space $\Tot \cO_{\mathbb{P}^n}(-1)$ with projection $\pi$. Then $\resolution$ is cut out by $\pi^* \cO_{\mathbb{P}^n}(n)$, and we find using the adjunction formula that $\omega_\resolution$ is trivial. As $\singular$ is a hypersurface in affine space, $\omega_\singular$ is also trivial, so indeed $\resMap$ is crepant, and the assumptions of Section~\ref{sect.setting} are satisfied. 
\end{eg}

\begin{rem}\label{eg.relaxed} We allow $\singular$ to be reducible in Example~\ref{eg.cone} above, for instance by taking the cone on a nodal curve $\{x_1x_2=0\}$ in $\mathbb{P}^2$. Then $\singular$ has two irreducible components, corresponding to the branches of the node. The blowup $\resolution$ of $\singular$ likewise has two irreducible components, namely the preimages of the branches under $\pi$. 

In this case therefore the morphism $\resMap\colon\resolution \to\singular$ is between non-normal varieties, as both are singular in codimension~$1$. Our conventions allow such contractions (though note that contractions may be assumed to be normal elsewhere in the literature).
\end{rem}

Recall that $\blowupAmbientMap$ is the blowup of smooth $\ambient$ in smooth $\blowupLocusA$ of constant codimension~$n+1$.
\begin{equation*}
    \begin{tikzpicture}[scale=\stdshrink]
    	\node (ambient) at (0,0) {$\ambient$};
	\node (blowupLocusA) at (-1,0) {$\blowupLocusA$};
	\node (resolutionAmbient) at (0,1) {$\resolutionAmbient$};
	\node (excBlowupAmbient) at (-1,1) {$\overset{\phantom{.}}{\excBlowupAmbient}$};
	\draw [right hook->] (blowupLocusA) to node[below]{\stdstyle $\blowupIncAmbient$} (ambient);
	\draw [right hook->] (excBlowupAmbient) to node[above]{\stdstyle $\blowupExcIncAmbient$} (resolutionAmbient);
	\draw [->] (excBlowupAmbient) to node[left]{\stdstyle $\blowupProjAmbient$} (blowupLocusA);
	\draw [->] (resolutionAmbient) to node[right]{\stdstyle $\blowupAmbientMap$} (ambient);
	\end{tikzpicture}
\end{equation*}
Then we have semiorthogonal decompositions~\cite{Orl}
\begin{equation}\label{eq.sods}
\D(\resolutionAmbient) = \big\langle \blowupAmbientMap^* \D(\ambient), \SODembedA_1 \D(\blowupLocusA), \dots, \SODembedA_{n} \D(\blowupLocusA) \big\rangle = \big\langle \SODembedA_{-n+1} \D(\blowupLocusA), \dots, \SODembedA_{0} \D(\blowupLocusA), \blowupAmbientMap^* \D(\ambient) \big\rangle
\end{equation}
with embeddings $\SODembedA_\index = \blowupExcIncAmbient[*] ( \blowupProjAmbient^*(\placeholder) \otimes \cO_\blowupProjAmbient(\index-1) )$.

\begin{prop}\label{lem.intertwine}
$\SODembedA_\index (\placeholder) \otimes \ideal_\resolution \cong \SODembedA_{\index-n} \comp (\placeholder\otimes \blowupIncAmbient^* \ideal_\singular)$
\end{prop}
\begin{proof}Using the projection
formula
\begin{align*}
 \SODembedA_\index (\placeholder)\otimes \ideal_\resolution& \cong \blowupExcIncAmbient[*] (\blowupProjAmbient^*(\placeholder) \otimes \cO_\blowupProjAmbient(\index-1) \otimes \blowupExcIncAmbient^* \ideal_\resolution )
\end{align*}
 but by Proposition~\ref{lem.ideals}(\ref{lem.idealsB})
 \begin{equation*}
 \blowupExcIncAmbient^* \ideal_\resolution \cong \blowupExcIncAmbient^*(\blowupAmbientMap^*\ideal_\singular \otimes \cO(n\excBlowupAmbient))  \cong  \blowupProjAmbient^* \blowupIncAmbient^*\ideal_\singular
\otimes \cO_\blowupProjAmbient(-n)\end{equation*}
and the claim follows.
\end{proof}

\begin{thm}[\optbracket{Theorem~\ref{thm.pullbacksph}}]\label{thm.main} In the setting of Section~\ref{sect.setting},
\begin{enumerate}[start=0,label={(\arabic*)},ref={\arabic*}]
\item\label{thm.mainA} $\compMap^*$ is spherical.
\end{enumerate}
Let $\iun$ be the restriction of $\incResInBlowupAmbient^*$ to the left orthogonal of $\blowupAmbientMap^* \D(\ambient)$ in $\D(\resolutionAmbient)$, as follows.
\beq
\iun = \incResInBlowupAmbient^* \colon {}^\perp \blowupAmbientMap^* \D(\ambient)  \to \D(\resolution)
\eeq
Then furthermore:
\begin{enumerate}[resume*]
\item\label{thm.mainpB} $\iun$ is spherical.
\item\label{thm.mainpC} There is an isomorphism
\begin{equation*}
\tw[\compMap^*](\placeholder) \otimes \normal{\resolution} \cong \tw[\iun]^{-1}(\placeholder)[2]
\end{equation*}
 between autoequivalences of $\D(\resolution)$.
\end{enumerate}
Here $\normal{\resolution}$ is the invertible normal sheaf of $\resolution$ in $\resolutionAmbient=\blowupAmbient$.
\end{thm}

\begin{proof}
By Proposition~\ref{prop.smCar2}(\ref{prop.smCarA}) we have that $\resolution$ is a Cartier divisor in $\resolutionAmbient$, so  
\begin{equation*}
\fun = \incResInBlowupAmbient^*\colon \D(\resolutionAmbient) \to \D(\resolution)
\end{equation*}
is spherical with cotwist $\placeholder\otimes \ideal_\resolution$ by Proposition~\ref{prop.divsph}(\ref{prop.divsphA}). By~\eqref{eq.sods}, $
\D(\resolutionAmbient) = \big\langle \mathsf{A}, \mathsf{B} \big\rangle = \big\langle \mathsf{B}', \mathsf{A} \big\rangle$
with $\mathsf{A}=\blowupAmbientMap^* \D(\ambient)$ and:
 \begin{equation*}\mathsf{B}  = \big\langle \SODembedA_1 \D(\blowupLocusA), \dots, \SODembedA_n \D(\blowupLocusA) \big\rangle
\qquad
\mathsf{B}'  = \big\langle \SODembedA_{-n+1} \D(\blowupLocusA), \dots, \SODembedA_{0} \D(\blowupLocusA) \big\rangle
\end{equation*}
By Proposition~\ref{lem.intertwine}, $\SODembedA_\index \D(\blowupLocusA)\otimes \ideal_\resolution  = \SODembedA_{\index-n} \D(\blowupLocusA)$
and so, as $\mathsf{B}$ and~$\mathsf{B}'$ are the smallest full triangulated subcategories containing the indicated images of~$\D(\blowupLocusA)$, we find the following.
\beq\ctw[\fun]\mathsf{B}=\mathsf{B}\otimes \ideal_\resolution = \mathsf{B}'\eeq

Then applying Theorem~\ref{thm.sodsph} to~$\fun$ gives spherical  
$\fun_{\mathsf{A}}= \fun \comp \blowupAmbientMap^* = \incResInBlowupAmbient^* \blowupAmbientMap^* \cong \compMap^*$ 
and~$\fun_{\mathsf{B}}=\iun$, yielding~(\ref{thm.mainA}) and~(\ref{thm.mainpB}) respectively, and furthermore $\tw[\fun] \cong \tw[\fun_{\mathsf{A}}] \comp \tw[\fun_{\mathsf{B}}]$ which yields~(\ref{thm.mainpC}) using Proposition~\ref{prop.divsph}(\ref{prop.divsphA}).
\end{proof}

Below we prove Theorem~\ref{thm.mainintro}, which describes $\tw[\compMap^*]$ in the case when $n=1$. Later we prove Theorem~\ref{thm.mainintroB} which describes $\tw[\compMap^*]$ for general $n\geq 1$ under further assumptions. For these results we use the following functors~$\gun_\index$.

\begin{defn}\label{def.sphA}  Recall  notation as follows.
\begin{equation*}
    \begin{tikzpicture}[scale=\stdshrink]
    	\node (singular) at (0,0) {$\singular$};
	\node (resolution) at (0,1) {$\resolution$}; 
	\node (blowupLocusA) at (-1,0) {$\blowupLocusA$};
	\node (excBlowupSingular) at (-1,1) {$\overset{\phantom{.}}{\excBlowupSingular}$};
	\draw [right hook->] (blowupLocusA) to (singular);
	\draw [->] (resolution) to node[right]{\stdstyle $\resMap$} (singular);
	\draw [right hook->] (excBlowupSingular) to (resolution);
	\draw [->] (excBlowupSingular) to node[left]{\stdstyle $\blowupProjA$} (blowupLocusA);
	\end{tikzpicture}
\end{equation*}
Define a functor 
\begin{equation*}
	\begin{tikzpicture}[xscale=\compdiagscale]
		\node (A) at (-0.25,0) {$ \gun_\index\colon\D(\blowupLocusA)$};
		\node (B) at (1,0) {$\D(\excBlowupSingular)$};
		\node (C) at (3,0) {$\D(\excBlowupSingular)$};
		\node (D) at (4,0) {$\D(\resolution)$};

		\draw [->] (A) to node[above]{\compstyle $\blowupProjA^*$} (B);
		\draw [->] (B) to node[above]{\compstyle $\otimes \cO_{\blowupProjA}(\index-1)$} (C);
		\draw [->] (C) to (D);
	\end{tikzpicture}
\end{equation*}
for $\index\in\mathbb{Z}$, where the last functor is  pushforward. 
\end{defn}

\begin{rem}\label{rem.sphobj} In the case that $\blowupLocusA$ is a point, as in Example~\ref{eg.cone}, we have $\gun_\index\colon \D(\pt) \to \D(\resolution)$ with $\gun_\index(\cO_{\pt})$ given by the following object.
\beq
\cE_\index =\cO_{\excBlowupSingular}(\index-1) \in \D(\resolution)
\eeq 
For $\gun_\index$ spherical, the twists $\tw[\gun_\index]$ are therefore simply twists $\twObj{\resolution}{\cE_\index}$ by spherical objects.
\end{rem}

\begin{prop}\label{lem.resembed}
$ \incResInBlowupAmbient^* \comp \SODembedA_\index \cong\gun_\index$

\end{prop}
\begin{proof} We have a fibre square as follows~\cite[Appendix~B.6.9]{Fulton}.
\begin{equation}\label{eq.twoexcsnotn}
\begin{aligned}
    \begin{tikzpicture}[xscale=\stdscale]
    	\node (excBlowupAmbient) at (1,0) {$\excBlowupAmbient$};
	\node (excBlowupSingular) at (0,0) {$\excBlowupSingular$};
	\node (resolutionAmbient) at (1,-1) {$\resolutionAmbient$};
	\node (resolution) at (0,-1) {$\resolution$};
	\draw [left hook->] (excBlowupSingular) to node[left]{\stdstyle$\incExcInResolution$}(resolution);
	\draw [right hook->] (resolution) to node[above]{\stdstyle $\incResInBlowupAmbient$} (resolutionAmbient);
	\draw [left hook->] (excBlowupAmbient) to node[right]{\stdstyle $\blowupExcIncAmbient$} (resolutionAmbient);
	\draw [right hook->] (excBlowupSingular) to node[above]{\stdstyle $\incExcInExcAmbient$}(excBlowupAmbient);
	
	\end{tikzpicture}
\end{aligned}
\end{equation}
For $\blowupProjAmbient\colon \excBlowupAmbient \to \blowupLocusA$ the projection, we have
\begin{equation}\label{eq.compat}\incExcInExcAmbient^*\cO_\blowupProjAmbient(1) \cong \cO_\blowupProjA(1)
\end{equation}
allowing us to rewrite $\gun_\index$ as follows.
 \begin{equation*}
\gun_\index  =  \incExcInResolution[*] (\blowupProjA^*(\placeholder) \otimes \cO_\blowupProjA(\index-1) )  \cong  \incExcInResolution[*] \incExcInExcAmbient^* ( \blowupProjAmbient^*(\placeholder) \otimes \cO_\blowupProjAmbient(\index-1) )   
\end{equation*}

To conclude, we claim base change for the intersection~\eqref{eq.twoexcsnotn}, namely the following.
\begin{equation}\label{eq.bcadd}
\incResInBlowupAmbient^* \blowupExcIncAmbient[*] \cong \incExcInResolution[*] \incExcInExcAmbient^*
\end{equation}
By Proposition~\ref{prop.smCar}(\ref{prop.smCarY}), $\resolutionAmbient$ and $\excBlowupAmbient$ are smooth equidimensional. Though $\resolution$ may be singular, it is Gorenstein by Proposition~\ref{prop.smCar2}(\ref{prop.smCarB}), and equidimensional by Proposition~\ref{prop.smCar2}(\ref{prop.smCarA}). Then the claimed isomorphism will follow by~\cite[Proposition~A.1]{Addington}. The reference assumes that $\resolution$,  $\excBlowupAmbient$  and $\resolutionAmbient$ are connected, but the argument goes through if they are equidimensional and $\excBlowupSingular$ is of the expected dimension, namely $\dim \resolution + \dim\excBlowupAmbient - \dim \resolutionAmbient$. For us, this is $\dim \resolution - 1$, and $\excBlowupSingular$ is the exceptional divisor of $\resolution$, so we are done.  
\end{proof}

\begin{rem}\label{rem.smoothness}
Assumption~\ref{assm.embed} may be weakened to allow~$\ambient$ smooth in a neighbourhood of~$\singular$. Then $\resolutionAmbient$ is smooth in a neighbourhood of $\resolution$
and of~$\excBlowupAmbient$, and the argument in~\cite[Proposition~A.1]{Addington} suffices to give the base change~\eqref{eq.bcadd} used above. It seems however that some smoothness of~$\ambient$ (giving smoothness of $\resolutionAmbient$) is needed for this base change, see~\cite[after proof of Proposition~A.1]{Addington}.

\end{rem}

\begin{thm}[\optbracket{Theorem~\ref{thm.mainintro}}]\label{thm.mainn1} In the setting of Section~\ref{sect.setting} suppose $n=1$.
Define a functor 
\begin{equation*}
	\begin{tikzpicture}[xscale=\compdiagscale]
		\node (A) at (-0.25,0) {$\gun\colon \D(\blowupLocusA)$};
		\node (B) at (1,0) {$\D(\excBlowupSingular)$};
		\node (C) at (2,0) {$\D(\resolution)$};

		\draw [->] (A) to node[above]{\compstyle $\blowupProjA^*$} (B);
		\draw [->] (B) to (C);
	\end{tikzpicture}
\end{equation*}
by composition, where the last functor is  pushforward. Then:

\begin{enumerate}[label={(\arabic*)},ref={\arabic*}]
\item $\gun$ is spherical.
\item There is an isomorphism
\begin{equation*}
 \tw[\compMap^*] (\placeholder)\otimes \normal{\resolution}  \cong  \tw[\gun]^{-1}(\placeholder)[2]
\end{equation*}
between autoequivalences of $\D(\resolution)$.
\end{enumerate}
\end{thm}

\begin{proof}
For $n=1$ we have ${}^\perp \blowupAmbientMap^* \D(\ambient) = \SODembedA_1 \D(\blowupLocusA)$, so $\iun$ is isomorphic to~$\gun_1$ by Proposition~\ref{lem.resembed}, and $\gun=\gun_1$ by construction, so the result is 
contained in Theorem~\ref{thm.main}.\end{proof}

For a first example, recall Example~\ref{eg.conifold}, namely a small resolution of a $3$-fold ordinary double point~$\{ac+bd=0\} \cong \{xy+z^2+w^2=0\}$. We also have the following.

\begin{eg} Let $\singular$ be a surface quadric cone $\{xy+z^2=0\}$, and $\blowupLocusA$ a line through the origin. This $\blowupLocusA$ is Cartier except at the origin, and blowup along it yields a minimal resolution of $\singular$. Note that $\excBlowupSingular$ here is not smooth, indeed it is  a nodal curve, the union of the strict transform of $\blowupLocusA$ and a $\P^1$ over the origin in $\singular$.
\end{eg}

\begin{rem} Continuing the pattern above by taking a curve $\{xy=0\}$ yields only degenerate examples. Indeed, we saw in 
Proposition~\ref{prop.smCar}(\ref{prop.smCarZ}) that $n\coloneqq\codim_\singular \blowupLocusA$ must be positive, but blowing up a smooth point gives an isomorphism, whereas blowing up a node $\blowupLocusA$ on a curve is not a contraction (because the exceptional fibre is two points).
\end{rem}

\section{Blowups in non-Cartier divisors} 
\label{sec.eg}
In this section we work in the setting of Theorem~\ref{keythm.blowupB}, where $\resMap$ is the blowup along a divisor. (Note that the divisor is non-Cartier precisely when our contraction~$\resMap$ is not an isomorphism.) Under the assumptions of Theorem~\ref{keythm.blowupB}, $\singular$~has a family of $3$-fold ordinary double points parametrized by a possibly singular and non-reduced~$\blowupLocusB$.

I examine the geometry of~the blowup of $\singular$ along $\blowupLocusA$ and prove the geometric part of Theorem~\ref{keythm.blowupB}, namely part (\ref{keythm.blowupBA}).  The remaining homological part of the proof is then completed in Section~\ref{section.divisor}. Further examples of the geometry are given in Section~\ref{section.Examples}.

\begin{rem} The blowup of $\singular$ along $\blowupLocusA$ may admit a flop: I use a duality which corresponds to this flop in examples, namely the relation between locally free sheaves $\bunsec$ and $\buntot$, see below Definition~\ref{def.dualsheaves} and Example~\ref{eg.conifoldagain}.
\end{rem}

Below I explain in detail the assumption of Theorem~\ref{keythm.blowupB}, part~(\ref{cor.blowupBgeomAssmA}), that the normal cone~$\normalcone{\blowupLocusA}{\singular}$ is cut out of the total space of $\normalof{\blowupLocusA}{\ambient}$ in a particular way. Note first that:

\begin{itemize}
\item $\normalof{\blowupLocusA}{\ambient}$ is a locally free sheaf of rank~$2$ using Proposition~\ref{prop.smCar}(\ref{prop.smCarQ}). 

\item $\normalof{\singular}{\ambient}$ is an invertible sheaf using Proposition~\ref{prop.smCar2}(\ref{prop.smCarA}).
\end{itemize}
Letting $\Tot\normalof{\blowupLocusA}{ \ambient}$ be the total space, and $\bunproj$ its projection to $\exbase$, we make the following.

\begin{assm}\label{assm.cutnormal} Take  a section $\sec$ of the rank~$2$ locally free sheaf \begin{equation*}
\bunsec = \sHom(\normalof{\blowupLocusA}{\ambient},\normalof{\singular}{\ambient}|_\blowupLocusA)
\end{equation*}
on $\blowupLocusA$. Assume that the normal cone $\normalcone{\blowupLocusA}{\singular}$ is cut out of $\Tot \normalof{\blowupLocusA}{\ambient}$ by the canonically induced section $\indsec$ of $\bunproj^* \normalof{\singular}{\ambient}|_\blowupLocusA$, as described in Proposition~\ref{prop.cutnormal}.
\end{assm}

The sheaves appearing in the above will be denoted as follows for readability. 

\begin{defn}\label{def.dualsheaves}
Notate sheaves on $\blowupLocusA$ as follows.
\begin{itemize}
\item $\buntot = \normalof{\blowupLocusA}{ \ambient}$ a  locally free sheaf of rank~$2$
\item $\bunlin=\normalof{\singular}{\ambient}|_\blowupLocusA$ an invertible sheaf
\end{itemize}
By construction we therefore have the following. 
\begin{equation}\label{eq.dualbuns}\bunsec \cong \sHom(\buntot,\bunlin)
\end{equation}
\end{defn}

\begin{prop}\label{prop.cutnormal}
 A section $\sec$ of $\bunsec$ canonically induces a section $\indsec$ of $\bunproj^* \bunlin$, writing the projection $\bunproj\colon\Tot\buntot \to \exbase$. If $\sec$ is regular, then $\indsec$ is also regular.
\end{prop}
\begin{proof} From \eqref{eq.dualbuns} we have
$\bunproj^* \bunsec \cong \sHom(\bunproj^* \buntot,\bunproj^* \bunlin)$
and $\bunproj^* \buntot$ has a tautological section, hence the section $\bunproj^* \sec$ of $\bunproj^* \bunsec$ induces a section~$\indsec$ of~$\bunproj^* \bunlin$ as required. 

If $\sec$ is regular, then $\indsec$ is regular by a local calculation as follows. Recalling that $\blowupLocusA$ is smooth, take local coordinates $\blowupLocusAcoord$, and fibre coordinates $\buntotcoord=(\buntotcoord_1,\buntotcoord_2)$ for $\Tot \buntot$. Then locally we have a section $\sec(\blowupLocusAcoord)$ of $\bunsec$, and may choose a trivialization of $\bunsec$ so that $\sec=(\sec_1,\sec_2)$ and  $\indsec=\sec\cdot \buntotcoord$ with the standard inner product. The claim is then that $\indsec$ is non-zero, which follows as the $\sec_i$ are non-zero by regularity.
\end{proof}

The following appears in the statement of Theorem~\ref{keythm.blowupB}.

\begin{defn}\label{defn.blowupLocusB} Let $\blowupLocusB=\zeroes{\sec}  \subset \blowupLocusA$
be the zeroes of the section $\sec$ from Assumption~\ref{assm.cutnormal}.
\end{defn}

If $\sec$ is regular as in Theorem~\ref{keythm.blowupB}, assumption~(\ref{cor.blowupBgeomAssmB}), then $\blowupLocusB$ is codimension~$2$ in~$ \blowupLocusA$. 

\begin{eg}\label{eg.conifoldagain} 

Take $\singular=\{ac+bd=0\}$ in~$\ambient = \mathbbm{k}^4$, and $\blowupLocusA = \{c,d=0\} \cong \mathbbm{k}^2$ with coordinates $(a,b)$. The bundle $\buntot=\normalof{\blowupLocusA}{\ambient}$ has fibre coordinates $(c,d)$, and $\singular$ is cut out of~$\ambient$ by a function so $\bunlin$ is trivial, and hence $\bunsec=\buntot^\vee$. Taking dual fibre coordinates for $\bunsec$, the regular section $\sec = (a,b)$  induces the section $\indsec=ac+bd$ of the trivial bundle~$\bunproj^* \bunlin$, which clearly cuts~$\normalcone{\blowupLocusA}{\singular}$ out of $\Tot\normalof{\blowupLocusA}{\ambient}$. Furthermore $\blowupLocusB =\zeroes{\sec}= \{ 0 \}$, the singular point of~$\singular$. 

Blowing up $\singular$ along $\blowupLocusA$ gives a small resolution $\resolution$. The duality between $\bunsec$ and $\buntot$ corresponds to the flop of this~$\resolution$, as follows. Identifying $\ambient$ with $\Tot\buntot$, the flop is given by blowing up $\singular$ along $\bunproj^{-1} \blowupLocusB$. Then $\bunproj^{-1} \blowupLocusB$ and $\blowupLocusA$ naturally arise as the zeroes of sections of~$\bunproj^*\bunsec$ and~$\bunproj^*\buntot$ respectively, namely $\bunproj^*\sec$ and the tautological section.
\end{eg}

\begin{prop}\label{prop.singloc} In the setting of Theorem~\ref{keythm.blowupB}, for the singular locus \,$\singLocus$ of $\singular$ we have:
\begin{equation*}
\singLocus \cap \blowupLocusA =\blowupLocusB
\end{equation*}
This may be interpreted as an equality of subschemes of $\singular$.
\end{prop}
\begin{proof} As in the proof of Proposition~\ref{prop.cutnormal} we work locally on $\blowupLocusA$. As there, take local coordinates $(\blowupLocusAcoord,\buntotcoord)$ for  $\Tot \buntot = \Tot \normalof{\blowupLocusA}{ \ambient}$ such that $\indsec=\sec(\blowupLocusAcoord)\cdot \buntotcoord$. Then to check the Jacobian criterion we may calculate as follows.
\beq \frac{d\indsec}{d(\blowupLocusAcoord,\buntotcoord)}= \frac{d(\sec\cdot \buntotcoord)}{d(\blowupLocusAcoord,\buntotcoord)}=\left(\frac{d\sec}{d\blowupLocusAcoord}\cdot \buntotcoord,\sec\cdot \frac{d\buntotcoord}{d\buntotcoord}\right)=\left(\frac{d\sec}{d\blowupLocusAcoord}\cdot \buntotcoord,\sec\right)
\eeq
The singular locus of the normal cone $\normalcone{\blowupLocusA}{\singular}$ is then locally cut out by these two functions and $\indsec$, therefore $\singLocus \cap \blowupLocusA$ is given by further cutting by~$\buntotcoord$. We find that $\singLocus \cap \blowupLocusA$ is locally cut out of $\Tot \normalof{\blowupLocusA}{ \ambient}$ by $\sec$ and $\buntotcoord$, giving~$\blowupLocusB$ scheme-theoretically, as required.
\end{proof}

This yields the following dichotomy.

\begin{prop}\label{prop.degen} In the setting of Theorem~\ref{keythm.blowupB}, we have:
\begin{enumerate}
\item\label{prop.degenA} If $\dim \singular \leq 2$ then $\blowupLocusB$ is empty.
\item\label{prop.degenB} If $\dim \singular \geq 3$
then $\singular$ is singular.
\end{enumerate}
\end{prop}
\begin{proof} The subscheme $\blowupLocusB$ has  constant codimension $2$ in $\blowupLocusA$, thence constant codimension $3$ in $\singular$. So if $\dim \singular \geq 3$ then $\blowupLocusB$ is non-empty, giving (\ref{prop.degenB}) using Proposition~\ref{prop.singloc}.
\end{proof}

The following is a duality property of $\bunsec$ and $\buntot$ related as in \eqref{eq.dualbuns}.

\begin{lem}\label{lem.equiv} For $\bunsec$ and $\buntot$ locally free sheaves of finite rank, and an invertible sheaf $\bunlin$, all on the same space, we have the following.
\begin{align*}
\bunsec \cong \sHom(\buntot,\bunlin) \quad & \Longleftrightarrow \quad \buntot \cong \sHom(\bunsec,\bunlin)
\end{align*}
If furthermore $\bunsec$ and $\buntot$ have rank~$2$, then the sheaves below are dual.
\begin{equation*}
\sHom(\det \bunsec,\bunlin) \qquad  \qquad \sHom(\det \buntot,\bunlin)
\end{equation*}
\end{lem}
\begin{proof} The first is clear. For the second, $\sHom(\det \bunsec,\bunlin)$ is dual to
\begin{equation*}
\sHom(\bunlin,\det \bunsec)  \cong \sHom(\bunlin^{\otimes 2},\det \bunsec) \otimes \bunlin  \cong \sHom ( \det \buntot , \bunlin)  
\end{equation*}
giving the result.\end{proof}

For Theorem~\ref{keythm.blowupB}(\ref{keythm.blowupBA}), we want to show that $\blowupProjA\colon \excBlowupSingular\to\blowupLocusA$ is the blowup of  $\blowupLocusA$ along the codimension~$2$ locus $\blowupLocusB$, which is cut out by a section of a rank~$2$ locally free sheaf $\bunsec$. The following gives a standard description of such blowups.

\begin{lem}\label{lem.projbunblow} For $ \bunsec$ a rank~$2$ locally free sheaf on $ \exbase$, write $\bunsecproj\colon \mathbbm{P} \bunsec \to \exbase$ for the associated projective bundle. Then we have that:

\begin{enumerate}
\item\label{lem.projbunblowX}
There is an isomorphism
\begin{equation*}
\bunsecproj_*(\cO_\bunsecproj(1) \otimes \bunsecproj^* \det \bunsec) \cong \bunsec^\vee \otimes \det \bunsec \cong \bunsec 
\end{equation*}
and thence a section~$\sec$ of $ \bunsec$ induces a section~$\blowsec$ of $\cO_\bunsecproj(1) \otimes \bunsecproj^* \det \bunsec$
on $ \mathbbm{P} \bunsec$. 
\end{enumerate}

\noindent Assuming furthermore that $\sec$ is regular:

\begin{enumerate}[resume*]
\item\label{lem.projbunblowA} $\zeroes{\blowsec} \cong \Bl_{\zeroes{\sec}}\blowupLocusA$.
\item\label{lem.projbunblowB} Restricting $\bunsecproj$ to $\zeroes{\blowsec}$ gives the blowup morphism.
\end{enumerate}
\end{lem}
\begin{proof} For (\opt{ams}{\ref{lem.projbunblowX}}\opt{comp}{i}), the first isomorphism uses the projection formula, and the second follows using that $\bunsec$ is rank~2. Then (\opt{ams}{\ref{lem.projbunblowA}}\opt{comp}{ii}) and (\opt{ams}{\ref{lem.projbunblowB}}\opt{comp}{iii}) follow by construction of the blowup.
\end{proof}

The following  will be used to compare the blowup of $\blowupLocusB$ in $\blowupLocusA$ with the blowup of $\blowupLocusA$  in $\singular$.

\begin{prop}\label{lem.projbun} Let $\bunsec$ and $\buntot$ be rank~$2$ locally free sheaves on $\exbase$, and $\bunlin$~an invertible sheaf on $\exbase$ such that $\bunsec \cong \sHom(\buntot,\bunlin).$ Then:
\begin{enumerate}
\item\label{lem.projbunA} There is an isomorphism $\mathbbm{P} \bunsec\cong \mathbbm{P}
 \buntot$ on $\exbase$ coming from the following.
\begin{align*}
\bunsec & \cong  \buntot\otimes \sHom(\det \buntot,\bunlin) \end{align*}

\item\label{lem.projbunB} Writing $\bunsecproj\colon \mathbbm{P} \bunsec \to \exbase$ and $\buntotproj\colon \mathbbm{P} \buntot \to \exbase$, we have corresponding invertible sheaves under the isomorphism of~{(\ref{lem.projbunA})} as follows.
\begin{equation*}
\cO_\bunsecproj(1) \otimes \blowupProjLocal^* \sHom(\det \buntot,\bunlin) \qquad\longleftrightarrow\qquad \cO_{\buntotproj}(1) 
\end{equation*}

\item\label{lem.projbunC} We have further corresponding sheaves as follows.
\begin{equation*}
\cO_\bunsecproj(1) \otimes \bunsecproj^* \det \bunsec \qquad\longleftrightarrow\qquad \cO_{\buntotproj}(1)\otimes \buntotproj^* \bunlin\end{equation*}
\end{enumerate}
\end{prop}

\begin{proof} 
\firstproofpart{\opt{ams}{1}\opt{comp}{i}}  Using that $\buntot$ is rank~$2$, we have
\begin{equation*}
\bunsec \cong \sHom(\buntot, \bunlin) \cong \sHom(\sHom(\buntot, \det \buntot), \bunlin) \cong  \buntot\otimes\buntwist\end{equation*}
where $\buntwist=\sHom(\det \buntot,\bunlin)$. The latter is invertible, so projectivizing gives the claim.

\proofpart{\opt{ams}{2}\opt{comp}{ii}} We have corresponding invertible sheaves
\begin{equation*}
\cO_\bunsecproj(1) \qquad\longleftrightarrow\qquad
\cO_{\buntotproj}(1) \otimes \buntotproj^* \buntwist^\vee
\end{equation*}
and we use that the correspondence is monoidal.

\proofpart{\opt{ams}{3}\opt{comp}{iii}} By Lemma~\ref{lem.equiv}, $\buntwist^\vee \cong \sHom(\det \bunsec,\bunlin)$ and we again use the  monoidal property.
\end{proof}

We are now ready to prove the geometric part of Theorem~\ref{keythm.blowupB}: the homological part is completed in Section~\ref{section.divisor}. We take our usual assumptions from Section~\ref{sect.setting} with the exception of Assumption~\ref{assm.crep}, namely crepancy of $\resMap$.

\begin{thm}[\optbracket{Theorem~\ref{keythm.blowupB}, part~(\ref{keythm.blowupBA})}]\label{cor.blowupBgeom}  

Under Assumptions~\ref{assm.embed} and~\ref{assm.blowup} of Section~\ref{sect.setting} with~$n=1$, take furthermore Assumption~\ref{assm.cutnormal} with section $\sec$ regular. Then:

\begin{enumerate}[start=0,label={(\arabic*)},ref={\arabic*}]
\item\label{cor.blowupBgeomA} The projection $\blowupProjA\colon\excBlowupSingular\to\blowupLocusA$ is the blowup of $\blowupLocusA$ along $\blowupLocusB$.
\end{enumerate}
Writing $\blowupProjB\colon \excBlowupB\to\blowupLocusB$ for the projection,  furthermore:
\begin{enumerate}[resume*]
\item\label{cor.blowupBgeomB} The restriction of \,$ \cO_{\blowupProjA}(1)$ to $\excBlowupB$  is given by
$\cO_{\blowupProjB}(1) \otimes \blowupProjB^* \buntwist|_\blowupLocusB$
 with $\buntwist$ as follows.
\begin{equation*}
\buntwist  =\sHom(\det\normalof{\blowupLocusA}{ \ambient}, \normalof{\singular}{\ambient}|_\blowupLocusA) 
\end{equation*}
\end{enumerate}
\end{thm}

\begin{figure}[htb]
\begin{center}
\threefoldpicture{1}
\captionsized{Sketch of the blowup of $\singular$ along $\blowupLocusA$ in the setting of Theorem~\ref{keythm.blowupB}. Compare Figure~\ref{fig.singn1} which shows the geometry before the blowup. The locus $\excBlowupSingularShort=\excBlowupAmbient$ is a $\mathbb{P}^1$-bundle over $\blowupLocusA$. I illustrate fibres $\excBlowupSingularShort_z$ over $z$ in~$\blowupLocusB$, and $\excBlowupSingularShort_y$ over $y$ not in $\blowupLocusB$. The normal bundle $\normalof{\excBlowupSingularShort}{\resolutionAmbient}$ is the total space of $\cO_{\blowupProjAmbient}(-1)$. This bundle contains the strict transform of\opt{ams}{ }\opt{comp}{~}$\normalcone{\blowupLocusA}{\singular}$, shown by thickened lines: its intersection with $\excBlowupSingularShort$ is the locus $\excBlowupSingular$.}
\label{fig.blowupn1}
\end{center}
\end{figure}

\begin{proof}
\firstproofpart{0} Recall that by Definition~\ref{defn.blowupLocusB}, $\blowupLocusB$ is the zeroes of the section~$\sec$ of~$\bunsec$ on~$\blowupLocusA$. Then by Lemma~\ref{lem.projbunblow}, writing $\bunsecproj\colon \mathbbm{P} \bunsec \to \exbase$, there  is an induced section $\blowsec$ of  $\cO_\bunsecproj(1) \otimes \bunsecproj^* \det \bunsec$ on $ \mathbbm{P} \bunsec$, and $\blowsec$ cuts out $\blowupB$. 

Now we compare $\blowupB$ with the exceptional locus~$\excBlowupSingular$, illustrated in Figure~\ref{fig.blowupn1}. Take  $\buntot= \normal{\blowupLocusA}{\ambient}$ and $\bunlin=\normal{\singular}{\ambient}|_\blowupLocusA$ as in Definition~\ref{def.dualsheaves}, and apply Proposition~\ref{lem.projbun}(\ref{lem.projbunA}). This gives the first isomorphism below: the second follows using Proposition~\ref{prop.smCar}(\ref{prop.smCarQ}).
\begin{equation}\label{eqn.isotoB}\P\bunsec \cong\P\buntot \cong \excBlowupAmbient
\end{equation}
I will show that $\blowupB$ is taken to $\excBlowupSingular$ by~\eqref{eqn.isotoB}.
First note that under \eqref{eqn.isotoB} we have, using Proposition~\ref{lem.projbun}(\ref{lem.projbunC}), corresponding invertible sheaves
\begin{equation*}
\cO_\bunsecproj(1) \otimes \bunsecproj^* \det \bunsec \qquad \longleftrightarrow \qquad  \cO_{\blowupProjAmbient}(1)\otimes \blowupProjAmbient^* \bunlin\end{equation*}
where we write $\blowupProjAmbient\colon \excBlowupAmbient \to \blowupLocusA$ for the projection. The section~$\blowsec$ of the left-hand sheaf corresponds to a section~$\excsec$, say, of the right-hand sheaf. Note that  
\begin{equation*}
\blowupProjAmbient_*(\cO_{\blowupProjAmbient}(1)\otimes \blowupProjAmbient^* \bunlin)  \cong  \sHom(\buntot,\bunlin) = \bunsec
\end{equation*}
and so $\sec$ also induces a section of $\cO_{\blowupProjAmbient}(1)\otimes \blowupProjAmbient^* \bunlin$. Reviewing the proofs of Lemma~\ref{lem.projbunblow}(\ref{lem.projbunblowX}) and Proposition~\ref{lem.projbun} 
this is found to be $\excsec$.

Now $\excBlowupSingular$ is the projectivization of $\normalcone{\blowupLocusA}{\singular}$. The latter lies in the bundle $\Tot\normalof{\blowupLocusA}{ \ambient}$. Writing $\bunproj$ for the bundle projection, Proposition~\ref{prop.cutnormal} explains that  $\sec$ induces a section~$\indsec$ of~$\bunproj^* \bunlin$ which, by Assumption~\ref{assm.cutnormal},  cuts out~$\normalcone{\blowupLocusA}{ \singular}$. This $\indsec$ is linear on the fibres of~$\bunproj$, and by the proof of Proposition~\ref{prop.cutnormal} we find that it induces the same section of $\cO_{\blowupProjAmbient}(1)\otimes \blowupProjAmbient^* \bunlin$ as above. Combining, the isomorphism \eqref{eqn.isotoB} restricts to~$\blowupB\cong\excBlowupSingular$, and the proof is completed using Lemma~\ref{lem.projbunblow}(\ref{lem.projbunblowB}).

\proofpart{\opt{ams}{1}\opt{comp}{i}} We have corresponding invertible sheaves under~\eqref{eqn.isotoB}  as follows, by Proposition~\ref{lem.projbun}(\ref{lem.projbunB}).
\begin{equation*}\cO_\blowupProjLocal(1) \otimes \blowupProjLocal^* \sHom(\det \buntot,\bunlin) \qquad \longleftrightarrow \qquad  
\cO_{\blowupProjAmbient}(1)
\end{equation*}
I claim that restricting to $\excBlowupB$ gives the following, yielding the result.
\begin{equation*}
 \cO_\blowupProjB(1) \otimes \blowupProjB^* \sHom(\det \buntot,\bunlin)|_\blowupLocusB \qquad \longleftrightarrow\qquad \blowupExcBInc^*\cO_{\blowupProjA}(1)
\end{equation*}
For the left-hand side note that the restriction of $\blowupProjLocal$ to $\excBlowupB$ is identified with~$\blowupProjB$, and thence the restriction of $\cO_\blowupProjLocal(1)$ is identified with $\cO_\blowupProjB(1)$. The claim follows by commutativity of:
\begin{equation*}\label{eq.commBb}
    \begin{tikzpicture}[scale=\stdshrink]
    	\node (blowupLocusA) at (0,0) {$\blowupLocusA$};
	\node (excBlowupSingular) at (0,1) {$\mathbbm{P} \bunsec$}; 
	\node (blowupLocusB) at (-1,0) {$\blowupLocusB$};
	\node (excBlowupB) at (-1,1) {$\excBlowupB$};
	\draw [right hook->] (blowupLocusB) to node[below]{} (blowupLocusA); 
	\draw [->] (excBlowupB) to node[left]{\stdstyle $\blowupProjB$} (blowupLocusB);
	\draw [right hook->] (excBlowupB) to node[above]{} (excBlowupSingular);
	\draw [->] (excBlowupSingular) to node[right]{\stdstyle $\bunsecproj$} (blowupLocusA); 
	\end{tikzpicture}
\end{equation*}
\noindent For the right-hand side we use the compatibility of $\cO_{\blowupProjA}(1)$ and $\cO_{\blowupProjAmbient}(1)$, namely~\eqref{eq.compat}.
\end{proof}

\begin{rem} The definition of $\buntwist$ in the theorem may be motivated as follows. In the degenerate case where $\blowupLocusA$ is Cartier in $\singular$, so that $\resMap$ is an isomorphism, then from
\begin{equation*}
0 \to \normalof{\blowupLocusA}{\singular} \to \normalof{\blowupLocusA}{\ambient} \to \normalof{\singular}{\ambient}|_\blowupLocusA \to 0
\end{equation*}
we find that $\buntwist \cong \conormalof{\blowupLocusA}{\singular}$. So we may think of $\buntwist$ as a substitute for $\conormalof{\blowupLocusA}{\singular}$ in the interesting case where $\blowupLocusA$ is non-Cartier in $\singular$ (and consequently $\normal{\blowupLocusA}{\singular}$ is not invertible).
\end{rem}

\begin{rem} Locally over~$\blowupLocusA$, the sections $\blowsec$ and~$\excsec$
in the proof above are given by
\beq
\bunseccoord_1 \,\bunsecproj^*\sec_2 -  \bunseccoord_2 \,\bunsecproj^*\sec_1  
\qquad \text{and} \qquad \buntotcoord_1 \,\blowupProjAmbient^*\sec_1 +  \buntotcoord_2 \,\blowupProjAmbient^*\sec_2  
\eeq
for fibre coordinates $\bunseccoord$ and $\buntotcoord$ on $\Tot\bunsec$ and $\Tot\buntot$ respectively, and taking a corresponding trivialization of $\bunsec$ so that $\sec=(\sec_1,\sec_2)$.
\end{rem}

Finally, we note how $\omega_\singular$ and $\omega_\blowupLocusB$ are related, for reference later.

\begin{prop}\label{prop.canonical} There is an isomorphism as follows.
\begin{equation*}
\omega_{\blowupLocusB} \cong  \omega_{\singular}\otimes \bunlin |_{\blowupLocusB} =  \omega_{\singular}\otimes \normalof{\singular}{\ambient} |_{\blowupLocusB} 
\end{equation*}
\end{prop}
\begin{proof} By the adjunction formula and Definition~\ref{def.dualsheaves} we have the following.
\beq
\omega_{\singular}|_{\blowupLocusA} \cong \omega_{\ambient}\otimes \normalof{\singular}{\ambient}|_\blowupLocusA = \omega_{\ambient}|_{\blowupLocusA} \otimes \bunlin
\eeq
Combining with
\beq
\omega_{\ambient}|_{\blowupLocusA} \cong \omega_{\blowupLocusA} \otimes \det \conormalof{\blowupLocusA}{ \ambient} = \omega_{\blowupLocusA} \otimes \det \buntot^\vee
\eeq
we then find
\beq
\omega_{\singular}|_{\blowupLocusA} \cong  \omega_{\blowupLocusA} \otimes \sHom(\det \buntot, \bunlin)  \cong  \omega_{\blowupLocusA} \otimes \sHom(\bunlin, \det \bunsec)
\eeq
where the second isomorphism is by Lemma~\ref{lem.equiv}. The result follows by restricting this to~$\blowupLocusB$ and comparing with
$\omega_{\blowupLocusB} \cong  \omega_{\blowupLocusA} \otimes \det \bunsec|_{\blowupLocusB}$ from the adjunction formula.\end{proof}

\begin{rem} In particular, in the case that $\bunlin$ is trivial on $\blowupLocusB$, if $\omega_\singular$ is trivial (or, equivalently by crepancy, $\omega_\resolution$ is trivial) then the same is true for~$\omega_\blowupLocusB$. In this case, therefore, we relate Calabi--Yau spaces of different dimension via spherical functors, see Remark~\ref{rem.dims}.  
\end{rem}

\section{Twists and non-Cartier divisors}
\label{section.divisor}

In this section I describe the twist $\tw[\compMap^*]$ for $\resMap$ the blowup in a non-Cartier divisor under  natural geometric assumptions, see Theorem~\ref{thm.blowupB} below. This is used to complete the proof of Theorem~\ref{keythm.blowupB}. To prove it, we repeat the trick of Theorem~\ref{thm.main}: there, we used that a pullback to $\resolution$ is spherical; here, we use that a pushforward to $\resolution$ is spherical, namely the pushforward from $\excBlowupSingular$. In this section, I write $\inc\colon \excBlowupSingular\into\resolution$. 

\opt{ams}{\stdskip}

Under the assumptions of Section~\ref{sect.setting} we have a blowup square as follows.
\begin{equation*}
    \begin{tikzpicture}[scale=1]
    	\node (singular) at (0,0) {$\singular$};
	\node (resolution) at (0,1) {$\resolution$}; 
	\node (blowupLocusA) at (-1,0) {$\blowupLocusA$};
	\node (excBlowupSingular) at (-1,1) {$\overset{\phantom{.}}{\excBlowupSingular}$};
	\draw [right hook->] (blowupLocusA) to  (singular);
	\draw [->] (resolution) to node[right]{\stdstyle $\resMap$} (singular);
	\draw [right hook->] (excBlowupSingular) to node[above]{\stdstyle $\inc$} (resolution);
	\draw [->] (excBlowupSingular) to node[left]{\stdstyle $\blowupProjA$} (blowupLocusA);
	\end{tikzpicture}
\end{equation*}

\begin{thm}\label{thm.blowupB} In the setting  of Section~\ref{sect.setting} take $n=1$, and require:

\begin{enumerate}[label={(\roman*)},ref={\roman*}]
\item\label{thm.blowupBassmA} The projection 
 $\blowupProjA$ is the blowup along $\blowupLocusB$ regularly embedded of  codimension~$2$.
 \end{enumerate}
 Writing $\blowupProjB\colon \excBlowupB\to\blowupLocusB$ for the projection, require also that:
 
\begin{enumerate}[resume*]\item\label{thm.blowupBassmB} The restriction of $ \cO_{\blowupProjA}(1)$ to $\excBlowupB$ is given by
$\cO_{\blowupProjB}(1) \otimes \blowupProjB^* \divcasetwist $
for $ \divcasetwist$~invertible.
\end{enumerate}

\noindent Now we may put
\begin{equation*}
	\begin{tikzpicture}[xscale=\compdiagscale]
		\node (A) at (-0.25,0) {$\hun\colon \D(\blowupLocusB)$};
		\node (B) at (1,0) {$\D(\excBlowupB)$};
		\node (C) at (2.75,0) {$\D(\excBlowupB)$};
		\node (D) at (3.75,0) {$\D(\resolution)$};

		\draw [->] (A) to node[above]{\compstyle $\blowupProjB^*$} (B);
		\draw [->] (B) to node[above]{\compstyle $\otimes \cO_{\blowupProjB}(-1)$} (C);
		\draw [->] (C) to (D);
	\end{tikzpicture}
\end{equation*}
where the last functor is  pushforward. Then:
\begin{enumerate}[label={(\arabic*)},ref={\arabic*}]
\item\label{thm.blowupBB} $\hun$ is spherical.
\item\label{thm.blowupBC} There is an isomorphism
\begin{equation*}
 \tw[\compMap^*](\placeholder)\otimes \resMap^* \normal{\singular}  \cong  \tw[\hun](\placeholder)  [2]
\end{equation*}
 between autoequivalences of $\D(\resolution)$.
\end{enumerate}
\end{thm}

\begin{rem} Note that $\resMap^* \normal{\singular}  = \resMap^* \omega_\incInAmbient \cong  \omega_\compMap   $ using crepancy of~$\resMap$ and that $  \omega_\compMap\cong \omega_\resMap \otimes \resMap^* \omega_\incInAmbient   $.
\end{rem}

I prove Theorem~\ref{thm.blowupB} at the end of this section. Combining with Theorem~\ref{cor.blowupBgeom} gives:

\begin{cor}[\optbracket{Theorem~\ref{keythm.blowupB}}]\label{cor.blowupB}
In the setting of Section~\ref{sect.setting} suppose $n=1$, and furthermore that Assumption~\ref{assm.cutnormal} holds with the section $\sec$ regular. Then:
\begin{enumerate}[start=0,label={(\arabic*)},ref={\arabic*}]
\item\label{cor.blowupBZ} The projection $\blowupProjA$ is the blowup of\, $\blowupLocusA$\! along $\blowupLocusB$, the zeroes of~$\sec$ in $\blowupLocusA$.
\end{enumerate}
Furthermore, parts~(\ref{thm.blowupBB}) and~(\ref{thm.blowupBC}) of Theorem~\ref{thm.blowupB} hold.
\end{cor}

Corollary~\ref{cor.blowupB} applies in a range of examples where $\singular$ contains a family of $3$-fold ordinary double points parametrized by $\blowupLocusB$, see Section~\ref{section.Examples}. A first interesting example is the following $3$-fold, which appears in Reid's pagoda paper~\cite{Pagoda}.

\begin{eg}\label{eg.pagoda} 
Let $\singular$ be $u^2-v^2=x^2-y^{2\pagoda}$ in $\ambient = \mathbb{A}^4$.  This has small resolutions $\resolution$ given by blowing up along $(u+v,x\pm y^\pagoda)$, so we take $\blowupLocusA = \{u+v,x+ y^\pagoda=0\}.$ 
Rewriting~$\singular$ as
\begin{equation*}
(u+v)\big((u+v)-2v\big)= (x+ y^\pagoda)\big((x+y^\pagoda)-2y^\pagoda\big)
\end{equation*}
the normal cone $\normalcone{\blowupLocusA}{\singular}$ is given by $(u+v)v= (x+ y^\pagoda)y^\pagoda$. Hence Assumption~\ref{assm.cutnormal} holds, as $\normalcone{\blowupLocusA}{\singular}$ is cut out of $\normalof{\blowupLocusA}{\ambient}$ by a function induced by a regular section $\sec=(v,-y^\pagoda)$ of the trivial rank~$2$ bundle on~$\blowupLocusA$. 

We are therefore in the setting of Corollary~\ref{cor.blowupB} with $\blowupLocusB= \{y^\pagoda,v=0\} \subset \blowupLocusA$, a fattened point of length~$\pagoda$. The spherical functor $\hun\colon\D(\blowupLocusB)\to\D(\resolution)$ may be described as follows. Let $R=\field[y]/y^\pagoda$ so that $\blowupLocusB = \Spec R$. Then $\cO_{\blowupProjB}(-1)$ on $\excBlowupB$ has an $R$-module structure via~$\blowupProjB$, and we may put $\cE=\cO_{\blowupProjB}(-1) \in \D(\resolution)$ using the embedding of $\excBlowupB$ in $\resolution$. Take a diagram as follows, where $\D(R)$ is the bounded derived category of finitely generated $R$-modules.
\begin{equation*}
	\begin{tikzpicture}[xscale=1.8]
		\node (A) at (0,0) {$\D(\blowupLocusB)$};
		\node (B) at (1,0) {$\D(\resolution)$};
		\node (C) at (0,-1) {$\D(R)$};
		\node (D) at (1,-1) {$\D(\resolution)$};

		\draw [->,transform canvas={yshift=-\arrowSep}] (A) to node[below]{\stdstyle $\hun$} (B);
		\draw [<-,transform canvas={yshift=\arrowSep}] (A) to node[above]{\stdstyle $\Radj\hun$} (B);

		\draw [->] (A) to node[below, sloped]{\stdstyle $\sim$} (C);
		\draw [->,transform canvas={yshift=-\arrowSep}] (C) to node[below]{\stdstyle $\cE{\otimes_R}\placeholder$} (D);
		\draw [<-,transform canvas={yshift=\arrowSep}] (C) to node[above]{\stdstyle $\Hom(\cE,\placeholder)$} (D);
		\draw [transform canvas={xshift=\equalsSep}] (B) to (D);
		\draw [transform canvas={xshift=-\equalsSep}] (B) to (D);

	\end{tikzpicture}
\end{equation*} 
Reviewing the definition of $\hun$, we see that both squares in the diagram commute, and thence $\tw[\hun]$ fits into a triangle as follows.
 \begin{equation*}
\cE\otimes_R\Hom(\cE,\placeholder) \to \id \to \tw[\hun] \to
\end{equation*}
We see that $\tw[\hun]$ recovers Toda's fat spherical twist~\cite{Tod}.\footnote{Toda takes  a sheaf $\cE$ on $\resolution \times \Spec R$. We equivalently take a sheaf $\cE$ on $\resolution$ with an $R$-module structure.}
However the method of proof here is quite different, being more geometric and avoiding homological calculations, and the isomorphism~\eqref{thm.blowupBC} is new to my knowledge.
\end{eg}

\begin{rem}\label{rem.pagoda} In the above Example~\ref{eg.pagoda} it is straightforward to check Corollary~\ref{cor.blowupB}(\ref{cor.blowupBZ}) directly as follows. Observe that~$\resolution$ is the graph of the rational map $\singular \dasharrow \mathbbm{P}^1$ given below.
\beq(u+v:x+ y^\pagoda)= (x- y^\pagoda:u-v) = \big((x +y^\pagoda) - 2y^\pagoda:(u+v)-2v\big)\eeq 
Now $\excBlowupSingular$ is the graph of this rational map after 
restriction to $\blowupLocusA$, \opt{ams}{which is }$(y^\pagoda:v)$. Furthermore $\blowupLocusA \cong \mathbb{A}^2$ where we take coordinates $(y,v)$ on the latter. Hence $\blowupProjA\colon \excBlowupSingular \to \blowupLocusA$ is the blowup of $\mathbb{A}^2$ along $\blowupLocusB= \{y^\pagoda,v=0\}$.  Note that~$\excBlowupSingular$ is singular for $\pagoda>1$: a chart on it is given by the graph of $y^\pagoda/v\colon\blowupLocusA \dasharrow \mathbbm{A}^1$, which is a singular hypersurface.
\end{rem}

\begin{rem}\label{rem.redconifold}
For $\pagoda=1$ in the above Example~\ref{eg.pagoda}, $\singular = \{ u^2-v^2=x^2-y^{2}\} \cong \{ ac+bd=0 \}$ after a change of coordinates, giving the 3-fold ordinary double point of Example~\ref{eg.conifoldagain}. In this case the argument above shows that $\tw[\hun]$ is simply a twist $\twObj{\resolution}{\cE}$ around a spherical object $\cE=\cO_{\excBlowupBshort}(-1)$ with support $\excBlowupBshort=\excBlowupB\cong\mathbb{P}^1$.
The isomorphism~\eqref{thm.blowupBC} then gives the following.
\begin{equation*}
 \tw[\compMap^*] \cong  \twObj{\resolution}{\cE} [2]
\end{equation*}
Note that this is easily seen to hold  on the object $\cE$ itself, as follows. The left-hand side gives~$\cE$ because $\resMap_* \cE = 0$ by a standard cohomology calculation, and $\tw[\compMap^*]$ fits in a triangle with $\compMap^*\compMap_* \to \id $ where $\compMap_*\cong\incInAmbient_*\resMap_*$. The right-hand side also gives~$\cE$ using that 
in general on a $3$-fold $\twObj{\resolution}{\cE} (\cE) \cong  \cE[-2]$, see for instance~\cite[Exercise~8.5(ii)]{Huy}.
\end{rem}

\begin{rem}\label{rem.BB} I explain the relation of Corollary~\ref{cor.blowupB} to work of Bodzenta and Bondal~\cite{BB}. They consider a flopping contraction of curves to $\singular$ an affine canonical hypersurface singularity of multiplicity~2. In this setting, there is an autoequivalence
\beq\funFlopFlop\colon \D(\resolution) \to\D(\resolution') \to \D(\resolution)
\eeq
via the derived category of the flop $\resolution'$. They explain that this autoequivalence fits in a triangle of Fourier--Mukai functors~\cite[end of Section~4.5, equation~65]{BB} as follows.\footnote{More precisely, they make this statement at the unbounded quasicoherent level.}
\beq
\id[1] \overset{\alpha}{\to} \compMap^* \compMap_* \to \funFlopFlop \to  
\eeq
Here $\normal{\singular}$ is trivial so $\compMap^! \cong \compMap^*[-1]$, hence $\alpha[-1]
$ yields a morphism $\eta\colon \id \to \compMap^! \compMap_*$. Supposing that $\eta$ is the unit, and that we may take functorial cones, then $\Cone(\eta)\cong\funFlopFlop[-1]$. On the other hand, using Remark~\ref{rem.inverses} we have $
\Cone(\eta)\cong\tw[\compMap^!]^{-1}[1]\cong\tw[\compMap^*]^{-1}[1] $.

\def\BBref{\cite[Corollary~5.18]{BB}}

Bodzenta--Bondal\opt{comp}{~\BBref} further give a spherical functor $\Psi$ such that $\tw[\Psi]^{-1} \cong \funFlopFlop$\opt{ams}{~\BBref}. Combining with the above we would find
$\tw[\compMap^*] \cong \tw[\Psi][2]$ which resembles the isomorphism~\eqref{thm.blowupBC} above. It would therefore be interesting to compare~$\Psi$ and our functor~$\hun$. More broadly, this resemblance suggests that there should be common generalizations of the results here and those of Bodzenta--Bondal.
\end{rem}

I now prepare for the proof of Theorem~\ref{thm.blowupB}. Take notation as follows.
\begin{equation}\label{eq.blowupsqB}
\begin{aligned}
    \begin{tikzpicture}[xscale=\stdscale]
    	\node (blowupLocusA) at (0,0) {$\blowupLocusA$};
	\node (excBlowupSingular) at (0,1) {$\excBlowupSingular$}; 
	\node (blowupLocusB) at (-1,0) {$\blowupLocusB$};
	\node (excBlowupB) at (-1,1) {$\excBlowupB$};
	\draw [right hook->] (blowupLocusB) to (blowupLocusA);
	\draw [->] (excBlowupB) to node[left]{\stdstyle $\blowupProjB$} (blowupLocusB);
	\draw [right hook->] (excBlowupB) to node[above]{\stdstyle $\blowupExcBInc$} (excBlowupSingular);
	\draw [->] (excBlowupSingular) to node[right]{\stdstyle $\blowupProjA$} (blowupLocusA);
	\end{tikzpicture}
	\end{aligned}
\end{equation}

\noindent Using this notation, assumption~(\ref{thm.blowupBassmB}) of Theorem~\ref{thm.blowupB} has the following form.

\begin{assm}[\optbracket{Assumption~\ref{assm.div}}]\label{assm.res} The restriction of $ \cO_{\blowupProjA}(1)$ to $\excBlowupB$ is  as follows.
 \begin{equation*} \blowupExcBInc^* \cO_\blowupProjA(1) \cong \cO_\blowupProjB(1) \otimes \blowupProjB^* \divcasetwist
\end{equation*}
\end{assm}

Though $\blowupLocusB$ may not be smooth or reduced (see for instance Example~\ref{eg.pagoda}) we still have: 

\begin{prop}\label{prop.sodsB} 
Given a square \eqref{eq.blowupsqB}
for the blowup of smooth $\blowupLocusA$ along $\blowupLocusB$ regularly embedded of  codimension~$2$, then:
\begin{enumerate}
\item\label{prop.sodsBA} Functors  $\blowupProjA^*$ and $\blowupProjB^*$ preserve the bounded coherent derived category.
\item\label{prop.sodsBB} There are semiorthogonal decompositions 
\begin{equation}\label{eqn.sodB}
\D(\excBlowupSingular) = \big\langle \blowupProjA^* \D(\blowupLocusA), \SODembedB_1 \D(\blowupLocusB) \big\rangle = \big\langle  \SODembedB_0 \D(\blowupLocusB), \blowupProjA^* \D(\blowupLocusA) \big\rangle 
\end{equation}
with embeddings $
\SODembedB_\index = \blowupExcBInc[*] ( \blowupProjB^*(\placeholder) \otimes \cO_\blowupProjB(\index-1) )$.
\end{enumerate}
\end{prop}
\begin{proof} 
\firstproofpart{\opt{ams}{1}\opt{comp}{i}} We have the following.
\begin{itemize}
\item $\blowupProjA^*$ is pullback from smooth $\blowupLocusA$.\footnote{Note however that $\excBlowupSingular$ may be singular, as in Remark~\ref{rem.pagoda}.}
\item $\blowupProjB^*$ is pullback along a flat morphism.
\end{itemize}
Note that the flatness here follows because $\blowupLocusB$ is regularly embedded in $\blowupLocusA$.

\proofpart{\opt{ams}{2}\opt{comp}{ii}} This is Orlov's blowup formula, which holds as $\blowupLocusB$ is regularly embedded. See \cite[Theorem 3.4]{KuzRat} for a sketch proof, and also~\cite[Corollary~6.10]{BerghSchnurer} in the setting of algebraic stacks (noting that their locally bounded pseudo-coherent categories coincide with bounded coherent categories assuming our schemes are Noetherian, which holds by finite type).
\end{proof}

\begin{prop}\label{eq.blowupBshift} Under Assumption~\ref{assm.res} we have the following.
\begin{equation*} \SODembedB_\index (\placeholder)\otimes \cO_\blowupProjA(-1) \cong\SODembedB_{\index-1} \comp  (\placeholder \otimes \divcasetwist^\vee )
\end{equation*}
\end{prop}
\begin{proof}
For the case $\index=1$ we have
 \begin{equation*}
 \blowupExcBInc[*] \blowupProjB^*(\placeholder)  \otimes \cO_\blowupProjA(-1)  \cong \blowupExcBInc[*] ( \blowupProjB^*(\placeholder\otimes \divcasetwist^\vee) \otimes \cO_\blowupProjB(-1) )
\end{equation*}
and the general case follows similarly.
\end{proof}

We define functors $\hun_\index$ for $\index\in\mathbb{Z}$ as follows, noting that $\hun_0$ gives $\hun$ from Theorem~\ref{thm.blowupB}.

\begin{defn}\label{defn.Gs}
Recalling the notation of~\eqref{eq.blowupsqB}, define a functor
\begin{equation*}
	\begin{tikzpicture}[xscale=\compdiagscale]
		\node (A) at (-0.25,0) {$\hun_\index\colon \D(\blowupLocusB)$};
		\node (B) at (1,0) {$\D(\excBlowupB)$};
		\node (C) at (3,0) {$\D(\excBlowupB)$};
		\node (D) at (4,0) {$\D(\resolution)$};

		\draw [->] (A) to node[above]{\compstyle $\blowupProjB^*$} (B);
		\draw [->] (B) to node[above]{\compstyle $\otimes \cO_{\blowupProjB}(\index-1)$} (C);
		\draw [->] (C) to (D);
	\end{tikzpicture}
\end{equation*}
by composition. Here $\blowupProjB^*$ preserves the bounded coherent derived category as noted in Proposition~\ref{prop.sodsB}(\ref{prop.sodsBA}), and the last functor is  pushforward.
\end{defn}

\begin{prop}\label{lem.G}$\inc_* \SODembedB_\index \cong \hun_\index$
\end{prop}
\begin{proof}Recalling that $\inc$ denotes the inclusion of $\excBlowupSingular$ in $\resolution$, the claim is clear.\end{proof}

\begin{prop}\label{lem.Grel} Under Assumption~\ref{assm.res} the functors $\hun_\index$ are related by:
\begin{equation*}
  \hun_{\index+l} (\placeholder) \cong \hun_\index (\placeholder\otimes \divcasetwist^{\vee\otimes l})\otimes \cO(-l\,\excBlowupSingular)
\end{equation*}
\end{prop}
\begin{proof} 
First note that $\inc^* \cO(\excBlowupSingular) \cong   \cO_\blowupProjA(-1)$ so Proposition~\ref{eq.blowupBshift} yields the following.
 \begin{equation*} \SODembedB_\index (\placeholder)\otimes \inc^* \cO(\excBlowupSingular) \cong \SODembedB_{\index-1} \comp  (\placeholder \otimes \divcasetwist^\vee )
\end{equation*}
Applying $\inc_*$ gives
\begin{equation*}
 \hun_\index(\placeholder)\otimes \cO(\excBlowupSingular) =  \hun_{\index-1} \comp (\placeholder \otimes \divcasetwist^\vee )
\end{equation*}
using Proposition~\ref{lem.G}. This gives the claim for $l=1$, and the rest follows by induction.
\end{proof}

Recall that we found a spherical functor $\gun$ from $\D(\blowupLocusA)$  in Theorem~\ref{thm.mainintro}.   We now establish that $\hun = \hun_0$ from $ \D(\blowupLocusB)$ is also spherical, and relate the twists of $\gun$ and $\hun$ on $\D(\resolution)$.

\begin{prop}\label{prop.blowupB}
Take the assumptions of Theorem~\ref{thm.blowupB} and recall the functors $\hun_\index$ from Definition~\ref{defn.Gs}. Then:
\begin{enumerate}[label={(\arabic*)},ref={\arabic*}]
\item\label{prop.blowupBA} $\hun_\index$ is spherical for each $\index\in\mathbb{Z}$.
\item\label{prop.blowupBB} There are isomorphisms
\begin{equation*}
\tw[\gun] \comp \tw[\hunb] \cong \placeholder\otimes  \cO(\excBlowupSingular) \cong \tw[\hun_0] \comp \tw[\gun] 
\end{equation*}
between autoequivalences of $\D(\resolution)$.
\end{enumerate}
\end{prop}

\begin{proof} Recall that $\gun$ is defined as follows.
\begin{equation*}
	\begin{tikzpicture}[xscale=\compdiagscale]
		\node (A) at (-0.25,0) {$\gun{} \colon \D(\blowupLocusA)$};
		\node (B) at (1,0) {$\D(\excBlowupSingular)$};
		\node (C) at (2.25,0) {$\D(\resolution)$};

		\draw [->] (A) to node[above]{\compstyle $\blowupProjA^*$} (B);
		\draw [->] (B) to node[above]{\compstyle $\inc_*$} (C);
	\end{tikzpicture}
\end{equation*}
Here $\excBlowupSingular$ is a Cartier divisor in $\resolution$ so, writing $\excBlowupSingularShort=\excBlowupSingular$ for readability,
\begin{equation*}\fun = \inc_*\colon \D(\excBlowupSingularShort) \to \D(\resolution)
\end{equation*}
 is spherical with cotwist $\placeholder\otimes \normal{\excBlowupSingularShort}[-2]$ by Proposition~\ref{prop.divsph}(\ref{prop.divsphB}). Now $\excBlowupSingularShort$ is the exceptional locus of the blowup $\blowupProjA$ so $\normal{\excBlowupSingularShort}=\cO_\blowupProjA(-1)$, and thence Proposition~\ref{eq.blowupBshift} yields the following.
 \beq
\ctw[\fun] \SODembedB_1 \D(\blowupLocusB)=\SODembedB_0 \D(\blowupLocusB)
 \eeq
Applying Theorem~\ref{thm.sodsph} to $\fun$ using the decompositions~\eqref{eqn.sodB} then gives spherical functors
\beq
\fun_{\mathsf{A}}  = \inc_*  \blowupProjA^* = \gun
\qquad
\fun_{\mathsf{B}}  =\inc_* \SODembedB_1  \cong \hunb
\eeq
where the isomorphism is by Proposition~\ref{lem.G}, yielding~(\ref{prop.blowupBA}) for $\index=1$.
 We further obtain a factorization $ \tw[\fun] \cong \tw[\fun_{\mathsf{A}}]\comp\tw[\fun_{\mathsf{B}}]\cong  \tw[\gun] \comp \tw[\hunb]$
and as $\tw[\fun] \cong \placeholder\otimes \ideal_{\excBlowupSingularShort}^\vee = \placeholder\otimes \cO(\excBlowupSingularShort)$
  by Proposition~\ref{prop.divsph}(\ref{prop.divsphB}), we get the first isomorphism of~(\ref{prop.blowupBB}).
  
  For the rest we use some standard facts about spherical functors. By  Proposition~\ref{lem.Grel}  we may write $\hun_\index\cong \Phi\hunb\!\Phi'$ with autoequivalences $\Phi$ and $\Phi'$ for each $\index\in\mathbb{Z}$. Knowing that~$\hun_1$ is spherical, this implies that $\hun_\index$ is also spherical, completing the proof of~(\ref{prop.blowupBA}).
  
To complete the proof of~(\ref{prop.blowupBB}), note that for $\index=0$ we have $\Phi = \placeholder\otimes \cO(\excBlowupSingularShort)$ and
\beq
\tw[\hun_0]  \cong\tw[\Phi\hunb\!\Phi'] \cong\tw[\Phi\hunb] \cong \Phi\tw[\hunb]\!\Phi^{-1}
\eeq
where the last two isomorphisms are straightforward consequences of the definitions: see for instance~\cite[Lemma~6.3]{Godinho}. We thence have
\begin{align*} \tw[\hun_0]
  & \cong \tw[\hunb] (\placeholder\otimes \cO(-\excBlowupSingularShort)) \otimes \cO(\excBlowupSingularShort)
\end{align*}
but using  the first isomorphism of~(\ref{prop.blowupBB}) we can write
\begin{equation*}  \tw[\hun_1](\placeholder) \cong \tw[\gun]^{-1}(\placeholder\otimes \cO(\excBlowupSingularShort)
)  
\qquad \Longrightarrow \qquad  \tw[\hun_0](\placeholder) \cong \tw[\gun]^{-1}(\placeholder)  \otimes \cO(\excBlowupSingularShort)
\end{equation*}
which rearranges to give the second isomorphism of~(\ref{prop.blowupBB}).
\end{proof}

\begin{rem}
The method of proof above is dual to that of the previous Theorem~\ref{thm.main} in the sense that it uses the spherical functor associated to the pushforward along the embedding of an effective Cartier divisor, rather than the pullback. However the argument here is simpler, as Proposition~\ref{lem.resembed} required base change whereas Proposition~\ref{lem.G} does not.
\end{rem}

Combining Theorem~\ref{thm.mainn1} with Proposition~\ref{prop.blowupB}, we now have that $\tw[\gun]$ is related to two different spherical twists, namely $\tw[\compMap^*]$ and $\tw[\hun]$. This finally gives the following.

\begin{proof}[Proof of Theorem~\ref{thm.blowupB}]
By Proposition~\ref{prop.blowupB} we have that $\hun =\hun_0 $ is spherical, and
\begin{align*}  \tw[\hun](\placeholder) & \cong \tw[\gun]^{-1}(\placeholder)  \otimes \cO(\excBlowupSingular) \\
& \cong \tw[\compMap^*] \comp (\placeholder) \otimes \normal{\resolution} \otimes \cO(\excBlowupSingular) [-2]
\tag{Theorem~\ref{thm.mainn1}} \\
 & \cong \tw[\compMap^*] (\placeholder)\otimes \resMap^*\normal{\singular} [-2] \tag{Proposition~\ref{lem.normals}}
\end{align*}
where in the last line we use that $n=1$, yielding the result.
\end{proof}

\section{Action on a spanning class}\label{sec.objects}

In this section I explain how our derived symmetries act on a spanning class. We make the following assumption on $\resMap$, which holds if $\resMap$ is a resolution of rational singularities.\footnote{We should take~$\singular$ normal here, see \cite[Remark~1.5]{Kov}.}

\begin{assm}\label{assm.rat}  The canonical morphism $\cO_\singular \to \resMap_* \cO_\resolution $ is an isomorphism.
\end{assm}

\noindent As usual, we take the derived pushdown. Note therefore that this strengthens our assumption that $\resMap$ is a contraction, which required only $\RDerived^0\resMap_*\cO_\resolution \cong \cO_\singular$.

\begin{prop} The objects $\mathcal{C} = \compMap^* \D(\ambient)\cup \ker \resMap_*$ 
are a (one-sided) spanning class.
\end{prop}
\begin{proof} 
We claim $\cF\in\mathcal{C}^\perp$  implies $\cF\cong 0$, where $\mathcal{C}^\perp $ is the right orthogonal to $\mathcal{C}$.\footnote{A (two-sided) spanning class also satisfies  this  condition  with ${}^\perp\mathcal{C}$ in place of $\mathcal{C}^\perp$, see for instance~\cite[Definition 1.47]{Huy}.} First, for $\cF\in\compMap^* \D(\ambient)^\perp$ we have $\compMap_*\cF\cong 0$ by adjunction and Yoneda. Now $\compMap_*\cF\cong\incInAmbient_*\resMap_*\cF\cong 0$ and, because $\incInAmbient$ is a closed embedding, we deduce that $\resMap_*\cF\cong 0$ and the result follows.
\end{proof}

\begin{prop}\label{prop.objA} In the setting of Section~\ref{sect.setting} we have:
\begin{enumerate}
\item\label{prop.objAA}
$\tw[\compMap^*]\compMap^*( \placeholder) \cong \compMap^*( \placeholder) \otimes \resMap^*\normal{\singular}^\vee[2]$
\item\label{prop.objAB}
$ \tw[\compMap^*] |_{\ker \resMap_*} \cong \id $
\item\label{prop.objAC}
$\hypertw \compMap^*( \placeholder) \cong \compMap^*( \placeholder) $
\item\label{prop.objAD}
$\hypertw  |_{\ker \resMap_*} \cong \placeholder \otimes \resMap^* \normal{\singular} [-2] $
\end{enumerate}
Furthermore, $\hypertw$ preserves $\ker \resMap_*$.
\end{prop}
\begin{proof}
By Definition~\ref{defn.hypertw} for the hypersurface twist $\hypertw$, (\ref{prop.objAC}) and (\ref{prop.objAD})  will follow from (\ref{prop.objAA})~and~(\ref{prop.objAB}) respectively. The last part then uses the
projection formula. For (\opt{ams}{\ref{prop.objAA}}\opt{comp}{i}), we have 
 \begin{align*} \tw[\compMap^*]\compMap^*( \placeholder)
 & \cong \compMap^*\ctw[\compMap^*] ( \placeholder) [2]
\end{align*}
by a general property of spherical functors, see for instance \cite[Section~1.3]{Addington}.\footnote{Note that Addington~\cite{Addington} writes $\ctw[\fun]$ where we would write $\ctw[\fun][1]$.}
Under Assumption~\ref{assm.rat}, $\id \to \resMap_*\resMap^*$ is an isomorphism by the projection formula, and so $\ctw[\compMap^*] \cong \ctw[\incInAmbient^*]$ using that $\compMap^* \cong  \resMap^* \incInAmbient^*$. But $\ctw[\incInAmbient^*] \cong \placeholder\otimes \ideal_\singular$ by  Proposition~\ref{prop.divsph}(\ref{prop.divsphA}). Combining we have
 \begin{align*} \tw[\compMap^*]\compMap^*( \placeholder)
 & \cong \compMap^* ( \placeholder\otimes \ideal_\singular) [2]\end{align*}
and we conclude using $\incInAmbient^* \ideal_\singular \cong \normal{\singular}^\vee$. Finally (\opt{ams}{\ref{prop.objAB}}\opt{comp}{ii}) is from the definition, using $\compMap_* \cong \incInAmbient_* \resMap_*$.
\end{proof}

\begin{cor} In the setting of Theorem~\ref{thm.mainintro} we have:
\begin{enumerate}
\item
$\tw[\gun]^{-1}\compMap^*( \placeholder) \cong \compMap^*( \placeholder) \otimes\cO(-\excBlowupSingular)$
\item
$ \tw[\gun]^{-1}|_{\ker \resMap_*} \cong \placeholder \otimes \normal{\resolution}[-2] $
\end{enumerate}
\end{cor}
 \begin{proof} Theorem~\ref{thm.mainintro} gives $\tw[\compMap^*](\placeholder) \otimes \normal{\resolution} \cong \tw[\gun]^{-1}(\placeholder) [2]$, and we use Propositions~\ref{prop.objA} and~\ref{lem.normals}.
 \end{proof}

\begin{cor} In the setting of Theorem~\ref{keythm.blowupB} we have:
\begin{enumerate}
\item
$\tw[\hun]\compMap^*( \placeholder) \cong \compMap^*( \placeholder)$
\item
$ \tw[\hun]|_{\ker \resMap_*} \cong \placeholder\otimes \resMap^* \normal{\singular}[-2] $
\end{enumerate}
Furthermore, $\tw[\hun]$ preserves $\ker \resMap_*$.
\end{cor}

\begin{proof} Theorem~\ref{keythm.blowupB} gives $\tw[\compMap^*](\placeholder)\otimes \resMap^* \normal{\singular}  \cong  \tw[\hun] (\placeholder) [2]$, and we again use Proposition~\ref{prop.objA}. 
\end{proof}

\section{Blowups in general codimension} 
\label{sect.gencodim}

In this section I prove Theorem~\ref{thm.mainintroB} which describes $\tw[\compMap^*]$ for  $\resMap$ the blowup in a locus with general codimension $n\geq 1$, under additional global assumptions.

\opt{ams}{\stdskip}

The following result is a cousin of Theorem~\ref{thm.sodsph} of Halpern-Leistner and Shipman, in a setting where the source of the spherical functor has a Serre functor~\cite{BondalKapranov,Huy}. 

\def\footnotetext{Addington and Aspinwall credit this result to unpublished work of Kuznetsov. A further treatment is given by Kuznetsov and Perry~\cite[Proposition A.1]{KP}.}

\begin{thm}[{\cite[Theorem~11]{AddingtonAspinwall}}\opt{ams}{\footnote{\footnotetext}}]\opt{comp}{\footnote{\footnotetext}}\label{thm.sodsphserre} Let  $\fun\colon \mathsf{D} \to \mathsf{D}'$ be a spherical functor where $\mathsf{D}$ has a Serre functor $\serre{}$ and a semiorthogonal decomposition as follows.
\begin{equation*}
\mathsf{D} = \big\langle \mathsf{A}_0, \dots, \mathsf{A}_n\big\rangle
\end{equation*}
Assume  $\ctw[\fun][d] \cong  \serre{}$ for some $d\in\mathbb{Z}$.
Then $\fun_\index$, the restriction of\,~$\fun$ to~$\mathsf{A}_\index$, is spherical, and:
\begin{equation*}
\tw[\fun] \cong \tw[\fun_0] \comp \dots \comp \tw[\fun_n]
\end{equation*}
\end{thm}

We may then prove the following.

\begin{thm}[\optbracket{Theorem~\ref{thm.mainintroB}}]\label{thm.mainB} In the setting of Section~\ref{sect.setting} suppose that:
\begin{enumerate}[label={(\roman*)},ref={\roman*}]
\item\label{thm.mainBaB} $\ambient$ is projective. 
\item\label{thm.mainBaA} $\singular$ is an anticanonical divisor in $\ambient$. 
\end{enumerate}
We may put 
\begin{equation*}
	\begin{tikzpicture}[xscale=\compdiagscale]
		\node (A) at (-0.25,0) {$ \gun_\index\colon\D(\blowupLocusA)$};
		\node (B) at (1,0) {$\D(\excBlowupSingular)$};
		\node (C) at (3,0) {$\D(\excBlowupSingular)$};
		\node (D) at (4,0) {$\D(\resolution)$};

		\draw [->] (A) to node[above]{\compstyle $\blowupProjA^*$} (B);
		\draw [->] (B) to node[above]{\compstyle $\otimes \cO_{\blowupProjA}(\index-1)$} (C);
		\draw [->] (C) to (D);
	\end{tikzpicture}
\end{equation*}
for $\index\in\mathbb{Z}$, where the last functor is  pushforward. Then:

\begin{enumerate}[label={(\arabic*)},ref={\arabic*}]
\item\label{thm.mainBrA} $\gun_\index$ is spherical for each $\index\in\mathbb{Z}$.
\item\label{thm.mainBrB} There is an isomorphism
\begin{equation*}
 \tw[\compMap^*] (\placeholder)\otimes \normal{\resolution} \cong \big(\tw[\gun_1] \comp \dots \comp \tw[\gun_n]\big)^{-1}(\placeholder)[2]
\end{equation*}
between autoequivalences of $\D(\resolution)$.\end{enumerate}
\end{thm}

\begin{proof}
We use an argument similar to Theorem~\ref{thm.main}. Recall from there that
\begin{equation*}
\fun = \incResInBlowupAmbient^*\colon \D(\resolutionAmbient) \to \D(\resolution)
\end{equation*}
is spherical with cotwist $\placeholder\otimes \ideal_\resolution$. Using (\ref{thm.mainBaB}) the blowup~$\resolutionAmbient$ is projective, and it is  smooth by Proposition~\ref{prop.smCar}(\ref{prop.smCarY}).  Hence $\D(\resolutionAmbient)$ has a Serre functor $\serre{} = \placeholder \otimes \omega_{\resolutionAmbient}[\dim {\resolutionAmbient}]$, see for instance \cite[Theorem~3.12]{Huy}. To check that $\serre{}$ satisfies the condition of Theorem~\ref{thm.sodsphserre}, first note that by (\ref{thm.mainBaA}) we have $\ideal_\singular \cong \omega_\ambient$. Then Proposition~\ref{lem.ideals}(\ref{lem.idealsX}) gives $ \ideal_\resolution \cong \blowupAmbientMap^!\ideal_\singular \cong \blowupAmbientMap^!\omega_\ambient \cong \omega_{\resolutionAmbient}$ 
so $\ctw[\fun]\cong \placeholder \otimes \omega_{\resolutionAmbient}$ and the condition is satisfied with $d=\dim \resolutionAmbient.$ 

Then applying Theorem~\ref{thm.sodsphserre} to the semiorthogonal decomposition
\begin{align}\label{eq.longsod}
\D(\resolutionAmbient) & = \big\langle \blowupAmbientMap^* \D(\ambient), \SODembedA_1 \D(\blowupLocusA), \dots, \SODembedA_n \D(\blowupLocusA) \big\rangle
\end{align}
from~\eqref{eq.sods}
gives spherical functors $\fun_\index$ for $\index=0,\dots,n$ and a factorization as follows.
\begin{equation*}
\tw[\fun] \cong \tw[\fun_0] \comp \tw[\fun_1] \comp \dots \comp \tw[\fun_{n}]
\end{equation*}
We have $\tw[\fun] \cong \placeholder\otimes \normal{\resolution}^\vee[2]$ and  $\fun_{0}\cong \compMap^*$ as in the proof of Theorem~\ref{thm.main}. Now $\fun_\index \cong \gun_\index$ for $\index=1,\dots,n$ using Proposition~\ref{lem.resembed}, and so these $\gun_\index$ are spherical.  The following Lemma~\ref{lem.resembeddiff} gives that for each $\index,\index'\in\mathbb{Z}$, $\gun_\index \cong \Phi \gun_{\index'}$ for some autoequivalence~$\Phi$, and we deduce~(\ref{thm.mainBrA}). For  (\ref{thm.mainBrB}), the factorization of $\tw[\fun]$ gives the following, and we rearrange. 
\begin{equation*}
\placeholder\otimes \normal{\resolution}^\vee [2]\cong  \tw[\compMap^*] \comp \tw[\gun_1] \comp \dots \comp \tw[\gun_n]\qedhere
\end{equation*}
\end{proof}

\begin{lem}\label{lem.resembeddiff}
 The functors $\gun_\index$ are related by invertible sheaves,  as follows.
\begin{equation*}
\gun_{\index+l}(\placeholder)\cong  \gun_\index (\placeholder) \otimes \cO_\resolution(-l\,\excBlowupSingular)    
\end{equation*}
\end{lem}

\begin{proof} We have $\gun_{\index+1} =  \incExcInResolution[*] (\blowupProjA^*(\placeholder) \otimes \cO_\blowupProjA(\index) )$
where $\incExcInResolution$ denotes the embedding of $\excBlowupSingular$ in~$\resolution$, as in the proof of Proposition~\ref{lem.resembed}. Noting that $\incExcInResolution^* \cO_\resolution(\excBlowupSingular) = \cO_\blowupProjA(-1)$ and using the projection formula we find
$ \gun_{\index+1} (\placeholder) \otimes \cO_\resolution(\excBlowupSingular)  \cong \gun_{\index} (\placeholder)$,
and the result follows.
\end{proof}

\begin{rem} The same argument using the other decomposition of~\eqref{eq.sods} gives the following.
\begin{equation*}
 \tw[\compMap^*] (\placeholder\otimes \normal{\resolution}) \cong \big( \tw[\gun_{-n+1}] \comp \dots \comp \tw[\gun_0]\big)^{-1}(\placeholder)[2]
\end{equation*}
\end{rem}

\begin{rem}\label{rem.relgen}Assumptions~(\ref{thm.mainBaB}) and~(\ref{thm.mainBaA}) of Theorem~\ref{thm.mainB} can surely be relaxed. In particular, we might replace the Serre functor in the proof above with a relative Serre functor by working over the base~$\ambient$, and noting that our contraction~$\resMap$ is projective by definition. The result could then follow from a relative version of Theorem~\ref{thm.sodsphserre}. Another approach could be to iteratively apply  Theorem~\ref{thm.sodsph} of Halpern-Leistner--Shipman, under appropriate assumptions, to the semiorthogonal decomposition~\eqref{eq.longsod}. 
\end{rem}

\begin{eg}\label{eg.quarticK3ext} Let $\singular$ be a quartic K3 surface in $\ambient = \mathbb{P}^3$ with a node~$x$. Let $\blowupLocusA = \{x\}$ and write $\excBlowupSingularShort=\excBlowupSingular\cong\mathbb{P}^1$ for the exceptional locus in the blowup~$\resolution$, as illustrated in Figure~\ref{fig.surf}. Then Theorem~\ref{thm.mainB} applies with $n=2$. As in Remark~\ref{rem.sphobj}, the twists $\tw[\gun_1]$ and $\tw[\gun_2]$ are simply twists  by spherical objects
\beq
\cO_{\excBlowupSingularShort} \cong \cO_{\mathbb{P}^1} \qquad\text{and}\qquad
\cO_{\excBlowupSingularShort}(1) \cong \cO_{\mathbb{P}^1}(2)\eeq 
using that $\excBlowupSingularShort$ is a conic in $\excBlowupAmbient\cong\mathbb{P}^2$. By~Theorem~\ref{thm.mainB}(\ref{thm.mainBrB}) we therefore have a relation
\begin{equation}\label{eq.twoterm}
 \tw[\compMap^*] (\placeholder)\otimes \normal{\resolution} \cong \big(\tw[\cO_{\mathbb{P}^1}] \comp \tw[\cO_{\mathbb{P}^1}(2)]\big)^{-1}(\placeholder)[2]
\end{equation}
in the autoequivalence group of $\D(\resolution)$. Note that $\normal{\resolution}|_{\mathbb{P}^1}\cong\cO_{\mathbb{P}^1}(4)$, where the~$4$ is minus the self-intersection number of the curve~$\excBlowupSingularShort$ in~$\resolution$. In Example~\ref{eg.defquarticK3} below I explain how, after some work, the above relation is compatible with known results.
\end{eg}

We have global analogues of the Calabi--Yau $n$-fold cones of Example~\ref{eg.cone}, as follows.

\begin{eg} As in Example~\ref{eg.cone}, consider the affine cone over a reduced hypersurface~$H$ of degree $n$ in $\mathbb{P}^n$ for $n\geq 2$, defined by $p_n(x)$ a homogeneous polynomial of degree~$n$ in variables $x_0,\dots,x_n$. Now we construct a singular Calabi--Yau $n$-fold, satisfying the assumptions of Theorem~\ref{thm.mainB}, with a chart given by a deformation of this cone. 

Take a further variable $y$ and put~$\singular=\zeroes{q} \subset \P^{n+1}$ for 
\beq
q(x,y) = p_n(x) y^2 + r_{n+1}(x)  y + r_{n+2}(x)
\eeq  
where $r_d(x)$ denotes a generic homogeneous polynomial of degree $d$ in the variables $x$. Then $\singular$ is reduced. On the chart $U=\{y\neq 0\}$, the restriction $\singular_U$~is cut out by 
\beq
q_U(x/y) = p_n(x/y) + r_{n+1}(x/y) + r_{n+2}(x/y)
\eeq  
after dividing through by $y^{n+2}$, so~$X_U$ is a deformation of the cone $\zeroes{p_n}$, as required. Take $\ambient=\P^{n+1}$ and $\blowupLocusA$ to be the point $(x:y)=(0:1)$. Then the normal cone $\normalcone{\blowupLocusA}{\singular}$ is  $\zeroes{p_n} \subset \A^{n+1}$ and so $\excBlowupSingular$ is a degree~$n$ hypersurface in $\excBlowupAmbient\cong\P^n$.

This satisfies the assumptions of Theorem~\ref{thm.mainB}: in particular, $\resMap$ is crepant by a local calculation on the chart~$U$ as in Example~\ref{eg.cone}, and $\singular$ is degree~$n+2$ in~$\P^{n+1}$ as required.
\end{eg}

\section{Compatibility with base change}
\label{section.Functoriality}

In this section I explain how the twist $\tw[\compMap^*]$ has a pleasing compatibility with base change. I show how this leads to interesting results, even for the basic example of a 3-fold ordinary double point, after base change to a hyperplane section.

\opt{ams}{\stdskip}

\begin{prop}\label{prop.twistbc} Take a fibre square where $\compMap$ and $\compMap'$ are obtained as in Section~\ref{sect.setting}.
\begin{equation*}
    \begin{tikzpicture}[scale=\stdshrink]
    	\node (ambientSource) at (1,0) {$\resolution$};
	\node (ambientSourceP) at (0,0) {$\resolution'$};
	\node (ambient) at (1,-1) {$\ambient$};
	\node (ambientP) at (0,-1) {$\ambient'$}; 
	\draw [->] (ambientSourceP) to node[left]{\stdstyle $\compMap'$}(ambientP);
	\draw [->] (ambientSourceP) to node[above]{\stdstyle $\ambientSourcebc$} (ambientSource);
	\draw [->] (ambientSource) to node[right]{\stdstyle $\compMap\phantom{'}$}(ambient);
	\draw [->] (ambientP) to (ambient);
	\end{tikzpicture}
\end{equation*} 
Then the associated twists are {intertwined} by $\ambientSourcebc_*$ as follows.
\beq
\tw[\compMap^*] \ambientSourcebc_* \cong \ambientSourcebc_* \tw[\compMap'^*]
\eeq
\end{prop}
\begin{proof}
This follows immediately from Lemma~\ref{lem.twistbc}(\ref{lem.twistbcB}) below, after noting the following. By assumption $\dim  \compMap = \dim \compMap' = -1$. By Theorem~\ref{thm.pullbacksph} the functors $\compMap^*$ and $\compMap'^*$ are spherical so their twists have inverses. Finally, $\resolution$ and $\resolution'$ are Gorenstein by Proposition~\ref{prop.smCar2}(\ref{prop.smCarB}).
\end{proof}

\begin{lem}\label{lem.twistbc} Take a fibre square of equidimensional schemes with $\ambient$~and~$\ambient' $ smooth, $\ambientSource$~and~$\ambientSource'$ Cohen--Macaulay, $\dim \compMap = \dim \compMap'$ and $\compMap$~proper, as follows.
\begin{equation*}
    \begin{tikzpicture}[scale=\stdshrink]
    	\node (ambientSource) at (1,0) {$\ambientSource$};
	\node (ambientSourceP) at (0,0) {$\ambientSource'$};
	\node (ambient) at (1,-1) {$\ambient$};
	\node (ambientP) at (0,-1) {$\ambient'$}; 
	\draw [->] (ambientSourceP) to node[left]{\stdstyle $\compMap'$}(ambientP);
	\draw [->] (ambientSourceP) to node[above]{\stdstyle $\ambientSourcebc$} (ambientSource);
	\draw [->] (ambientSource) to node[right]{\stdstyle $\compMap\phantom{'}$}(ambient);
	\draw [->] (ambientP) to  node[below]{\stdstyle $\ambientbc$}(ambient);
	\end{tikzpicture}
\end{equation*} 
Assume that $\tw[\compMap^*]$ and $\tw[\compMap'^*]$ fit into triangles of Fourier--Mukai functors as usual. Then:
\begin{enumerate}
\item\label{lem.twistbcA} $\ambientSourcebc^* \tw[\compMap^*] \cong \tw[\compMap'^*] \ambientSourcebc^* $
\end{enumerate}
If furthermore $\tw[\compMap^*]$ and $\tw[\compMap'^*]$ have inverses, then:
\begin{enumerate}[resume*]
\item\label{lem.twistbcB} $ \tw[\compMap^*] \ambientSourcebc_* \cong \ambientSourcebc_* \tw[\compMap'^*]  $
\end{enumerate}
\end{lem}
\begin{proof}
We have an isomorphism $ \ambientSourcebc^*\compMap^*\compMap_* \cong \compMap'^*\ambientbc^*\compMap_* \isoto\compMap'^*\compMap'_*\ambientSourcebc^*$ where the base change follows using the argument of \cite[Proposition~A.1]{Addington}, as in the proof of Proposition~\ref{lem.resembed}. We then check that this isomorphism fits into the following commutative square using the description of the base change morphism in~\cite[Proposition~3.7.2]{Lip}. 
\begin{equation*}
    \begin{tikzpicture}[xscale=1.25]
    	\node (A) at (1,0) {$\compMap'^*\compMap'_*\ambientSourcebc^*$};
	\node (B) at (0,0) {$\ambientSourcebc^*\compMap^*\compMap_*$};
	\node (C) at (1,-1) {$\ambientSourcebc^*$};
	\node (D) at (0,-1) {$\ambientSourcebc^*$}; 
	\draw [->] (B) to node[left]{\stdstyle $\ambientSourcebc^*\varepsilon$}(ambientP);
	\draw [->] (B) to node[above]{\stdstyle $\sim$} (A);
	\draw [->] (A) to node[right]{\stdstyle $\varepsilon'\ambientSourcebc^*$}(C);
	\draw [transform canvas={yshift=\equalsSep}] (D) to (C);
	\draw [transform canvas={yshift=-\equalsSep}] (D) to (C);
	\end{tikzpicture}
\end{equation*} 
Here $\varepsilon$ and $\varepsilon'$ denote counits. Then (\opt{ams}{\ref{lem.twistbcA}}\opt{comp}{i}) follows by forming triangles of Fourier--Mukai functors using the two vertical arrows, and (\opt{ams}{\ref{lem.twistbcB}}\opt{comp}{ii}) by taking adjoints and rearranging.
\end{proof}

\begin{eg}\label{eg.defquarticK3} I continue Example~\ref{eg.quarticK3ext}, where $\singular$ was a quartic K3 surface. Let $\singular$ now be a one-parameter \emph{deformation} of a quartic K3 surface~$\singular'$ with a node~$x$. Assume this deformation to be embedded in a smooth one-parameter deformation $\ambient$ of the ambient space  $\ambient'=\mathbb{P}^3$ of~$\singular'$. Require that $x$ is an ordinary double point in the 3-fold $\singular$. 

Assume given a smooth divisor $\blowupLocusA$ in $\singular$ whose blowup yields a small resolution $\resolution$ of $\singular$ whose central fibre is the resolution $\resolution'$ of $\singular'$ given by blowup of $\blowupLocusA' = \{x\}$. We are then in the setting of Theorem~\ref{thm.pullbacksph} for both $\resolution$ and $\resolution'$, and the conditions of Proposition~\ref{prop.twistbc} are satisfied with $\ambientSourcebc\colon\resolution'\into\resolution$ the inclusion. 

Now the surface $\singular'$ was studied in Example~\ref{eg.quarticK3ext}, where we found the isomorphism below, with $\mathbb{P}^1$ denoting the exceptional curve of $\resMap'$.
\begin{equation}\label{eq.surface}
\tw[\compMap'^*] (\placeholder)\otimes\normal{\resolution'} \cong \big(\tw[\cO_{\mathbb{P}^1}] \tw[\cO_{\mathbb{P}^1}(2)]\big)^{-1} (\placeholder)[2]
\end{equation}
I explain how~\eqref{eq.surface}  is consistent with known relations amongst autoequivalences of\opt{ams}{ }\opt{comp}{~}$\D(\resolution')$.

Write $\cE = \cO_{\P^1}(-1)$ in $\D(\resolution)$. It is well-known, and a consequence of Theorem~\ref{keythm.blowupB}, that this is a spherical object. Noting that $\normal{\resolution'} \cong \bcembed^*\normal{\resolution}$ we may relate $\tw[\compMap'^*]$ with $\twObj{\resolution}{\cE}$ via
\begin{align*}
\bcembed_* \big(\tw[\compMap'^*] (\placeholder)\otimes\normal{\resolution'}\big) \cong \bcembed_* \!\tw[\compMap'^*] (\placeholder)\otimes\normal{\resolution} 
& \cong \tw[\compMap^*] \!\bcembed_*(\placeholder)\otimes\normal{\resolution} \tag{Proposition~\ref{prop.twistbc}} \\
& \cong \twObj{\resolution}{\cE}    \bcembed_* (\placeholder)\otimes \resMap^* \normal{\singular}^\vee[2]\otimes\normal{\resolution}  \tag{Theorem~\ref{keythm.blowupB}}\\
& \cong \twObj{\resolution}{\cE}    \bcembed_* (\placeholder)\otimes \cO(-\excBlowupSingularShort) [2] \tag{Proposition~\ref{lem.normals}}  
\end{align*}
where I write $\excBlowupSingularShort$ 
for the exceptional divisor in $\resolution$.

On the other hand, $\cE' = \cO_{\P^1}(-1)$ in $\D(\resolution')$ is an example of a $\mathbb{P}^n$-object of Huybrechts and Thomas~\cite{HuyTho} for the case $n=1$. In this case their definition reduces to that of a spherical object on a variety of dimension~$2$. They find an intertwinement $\twObj{\resolution}{\cE}   \bcembed_* \cong \bcembed_* \!\twObj{\resolution}{\cE'}^2$ in \cite[Propositions~2.7 and~2.9]{HuyTho}. Combining with the above  gives
\begin{align*}\bcembed_* \big(\tw[\compMap'^*] (\placeholder)\otimes\normal{\resolution'}\big)[-2] & \cong   \bcembed_* \big(\twObj{\resolution}{\cE'}^2 (\placeholder)\otimes \bcembed^*\cO(-\excBlowupSingularShort) \big)\\
 & \cong   \bcembed_* \big(\twObj{\resolution}{\cE'}^2 (\placeholder)\otimes \cO(-2\excBlowupSingularShort') \big)
\end{align*}
where I write $\excBlowupSingularShort'$ 
for the exceptional curve in $\resolution'$. Here the coefficient~$2$ for~$\excBlowupSingularShort'$ comes from the degree of the conic~$\excBlowupSingularShort'$ in~$\excBlowupSingularShort\cong\P^2$.\footnote{Indeed, it is clear that $\bcembed^*\cO(\excBlowupSingularShort)\cong\cO(m\excBlowupSingularShort')$ for some $m\in\Z$. But restricting $\cO(\excBlowupSingularShort)$ to $\excBlowupSingularShort'\cong\P^1$ gives $\cO_{\mathbb{P}^1}(-2)$ as in Example~\ref{eg.quarticK3ext}, and we may deduce the claim.} We may then combine with \eqref{eq.surface} as follows.
\begin{align*}\label{eq.intertwin}
 \bcembed_*  \big(\tw[\cO_{\mathbb{P}^1}] \tw[\cO_{\mathbb{P}^1}(2)]\big)^{-1} (\placeholder) & \cong    \bcembed_* \big(\twObj{\resolution}{\cO_{\P^1}(-1)}^2 (\placeholder)\otimes \cO(-2\excBlowupSingularShort') \big)
\end{align*}
By Proposition~\ref{prop.K3rel}, this is implied by a known relation amongst autoequivalences of~$\D(\resolution')$.
\end{eg}

\begin{prop}\label{prop.K3rel}
Take a K3 surface~$B$ with a \mbox{$-2$-curve}~$C\cong\P^1$. Assume there exists a line bundle~$\cO_B(1)$ with degree~$1$ on $C$, so that $\cO(C) \cong \cO_B(-2)$. Then writing
\beq\twk[k] = \tw[\cO_C(k)] \text{\qquad and \qquad} \lb[k] =\placeholder\otimes\cO_B(k)\eeq
for autoequivalences of $\D(B)$, there is a relation $(\twk[0]\twk[2])^{-1} \cong  \lb[4] \twk[-1]^2$ so that we have:
\begin{align*}
( \twk[0]\twk[2])^{-1} (\placeholder) & \cong  \twk[-1]^2 (\placeholder) \otimes\cO(-2C)
\end{align*}
\end{prop}
 
 \begin{proof}
It is well known that $(\lb[1] \twk[-1])^2  \cong \id
$ as may be seen by an argument with a tilting bundle $\cO_B\oplus\cO_B(1)$: this bundle is tilting relative to the contraction of $C$, and $\lb[1] \twk[-1]$ swaps its two summands. Noting also that $\lb[l] \twk[k]\cong \twk[k+l]\lb[l] $ by for instance \cite[Lemma~8.21]{Huy}, we find
$ \twk[-1] \twk[0]\cong \lb[-2]$ and thence that
$\lb[4]\twk[-1]^2 \twk[0]\twk[2]  \cong \lb[4]\twk[-1]\lb[-2]\twk[2]  \cong  \lb[4]\twk[-1]\twk[0]\lb[-2]$ 
which is isomorphic to the identity, giving the result.
\end{proof}

\begin{rem} In the case of hypersurfaces of varieties with exceptional sequences, Canonaco and Karp~\cite{CanKarp,Can} have established general methods to obtain associated higher degree relations in the autoequivalence group, of a similar flavour to the relations above.\end{rem}

\section{Local examples}
\label{section.Examples}

In this final section, I give a specialization of the setting of Theorem~\ref{keythm.blowupB} in which $\ambient$ is the total space of a locally free sheaf over~$\blowupLocusA$. This may be thought of as a local model for the global setting of Sections~\ref{sec.eg} and~\ref{section.divisor}. I~finish with some examples of this geometry.

\begin{setn}[local model]\label{setn.local} Take sheaves on $\blowupLocusA$ smooth equidimensional as follows.
\begin{itemize}
\item $\buntot$ a locally free sheaf of rank $2$ 
\item $\bunlin$ an invertible sheaf
\end{itemize}
Let $\ambient= \Tot \buntot$ with projection~$\bunproj$ and take also the following.
\begin{itemize}
\item $\sec$ a regular section  of $\bunsec= \sHom(\buntot,\bunlin)$
\end{itemize}
For $\indsec$ the canonical regular section of $\bunproj^*\bunlin$ from Proposition~\ref{prop.cutnormal} let $\singular=\zeroes{\indsec}  \subset \ambient$ and consider $\blowupLocusA$ as the zero section in~$S$, so that $\blowupLocusA$ is contained in~$\singular$. Let $\resMap\colon\resolution\to\singular$ be the blowup of~$\singular$ along~$\blowupLocusA$. 
\end{setn}

\begin{prop}\label{prop.localThmC} Setting~\ref{setn.local} satisfies the assumptions of Theorem~\ref{keythm.blowupB}.
\end{prop}
\begin{proof}
We first check Assumptions~\ref{assm.embed} and~\ref{assm.blowup} of Section~\ref{sect.setting}. Note that $\singular$ is a hypersurface in smooth $\ambient$ by construction, and that $\codim_\ambient \blowupLocusA=\rk\buntot=2$ hence $n\coloneqq\codim_\singular \blowupLocusA=1$. For Assumption~\ref{assm.crep}, by Theorem~\ref{cor.blowupBgeom}(\ref{cor.blowupBgeomA}) the projection $\blowupProjA\colon\excBlowupSingular\to\blowupLocusA$ is itself the blowup of~$\blowupLocusA$ along~$\blowupLocusB$, whose exceptional locus is a divisor in~$\excBlowupSingular$. But $\blowupProjA$ is the restriction of $\resMap$, so then the exceptional locus of~$\resMap$ is codimension~$2$, hence $\resMap$ is crepant.

Assumption~\ref{assm.cutnormal} holds by construction, as the normal bundle $\normalof{\blowupLocusA}{ \ambient} \cong \buntot$ may be identified with the bundle $\bunproj\colon\ambient\to\blowupLocusA$ itself and, furthermore, noting that $\indsec$ is linear on fibres of $\bunproj$, the normal cone~$\normalcone{\blowupLocusA}{\singular}$ may be identified with $\singular$.
\end{proof}

Recall the following from the statement of Theorem~\ref{keythm.blowupB}.

\begin{defn}\label{defn.blowupLocusBlocal} Let 
$\blowupLocusB=\zeroes{\sec}  \subset \blowupLocusA$ be the zeroes of the section $\sec$ from Setting~\ref{setn.local}.
\end{defn}

I now specialize some of the results of Sections~\ref{sec.eg} and~\ref{section.divisor} to the local model Setting~\ref{setn.local}. In this setting the blowup square from Definition~\ref{def.sphA} acquires further morphisms induced by the projection~$\bunproj\colon\ambient \to \exbase$. These map from right to left in the  diagram below.
\begin{equation}\label{eq.blowupAlocal}
\begin{aligned}
    \begin{tikzpicture}[scale=\stdscale]
    	\node (singular) at (0,0) {$\singular$};
	\node (resolution) at (0,1) {$\resolution$}; 
	\node (blowupLocusA) at (-1,0) {$\blowupLocusA$};
	\node (excBlowupSingular) at (-1,1) {$\overset{\phantom{.}}{\excBlowupSingular}$};
	\draw [right hook->,transform canvas={yshift=\arrowSep}] (blowupLocusA) to (singular);
	\draw [<-,transform canvas={yshift=-\arrowSep}] (blowupLocusA) to (singular);
	\draw [->] (resolution) to node[right]{\stdstyle $\resMap$} (singular);
	\draw [right hook->,transform canvas={yshift=\arrowSep}] (excBlowupSingular) to (resolution);
	\draw [<-,transform canvas={yshift=-\arrowSep}] (excBlowupSingular) to (resolution);
	\draw [->] (excBlowupSingular) to node[left]{\stdstyle $\blowupProjA$} (blowupLocusA);
	\draw [->] (resolution) to node[above left]{\stdstyle $\bunprojres$} (blowupLocusA);
	\end{tikzpicture}
\end{aligned}
\end{equation}

We then have the following version of Theorem~\ref{keythm.blowupB} in the local model.

\begin{cor} In Setting~\ref{setn.local} we have that:
\begin{enumerate}[label={(\arabic*)},ref={\arabic*}]
\item $\hun$ is spherical.
\item There is an isomorphism
\begin{equation*}
 \tw[\compMap^*](\placeholder)\otimes \bunprojres^* \bunlin  \cong  \tw[\hun](\placeholder)  [2]
\end{equation*}
 between  autoequivalences of $\D(\resolution)$.
\end{enumerate}
Here $\bunprojres$ is the projection to~$\blowupLocusA$ from \eqref{eq.blowupAlocal} above.
\end{cor}
\begin{proof}
This is Theorem~\ref{keythm.blowupB} after noting that $\normal{\singular} = \normalof{\singular}{\ambient} \cong \bunproj^* \bunlin |_\singular$
 so that $\resMap^* \normal{\singular} \cong  \bunprojres^* \bunlin$.
 \end{proof}

As in Theorem~\ref{cor.blowupBgeom}(\ref{cor.blowupBgeomB}), we have that the  restriction of $ \cO_{\blowupProjA}(1)$ to $\excBlowupB$ is
$\cO_{\blowupProjB}(1) \otimes \blowupProjB^* \buntwist|_\blowupLocusB$
 for $\buntwist$ an invertible sheaf on $\blowupLocusA$. This $\buntwist$ may be described as follows.

\begin{prop}\label{prop.buntwist} 
In Setting~\ref{setn.local} we have that:
\beq
\buntwist  \cong \sHom(\det\buntot, \bunlin)  \cong \sHom(\bunlin, \det\bunsec)  
 \eeq
We have special cases as follows.
\begin{enumerate}
\item\label{prop.bunspecialA} If $\bunlin$ is trivial then
$\bunsec \cong \buntot^\vee$ and $\buntwist \cong \det\bunsec\cong \det \buntot^\vee.$
\item\label{prop.bunspecialB} If $\bunlin\cong\det \buntot$, or equivalently $\bunlin\cong\det \bunsec$, then
$ \bunsec \cong\buntot$ 
 and $\buntwist$ is trivial.
\end{enumerate}
\end{prop}

\begin{proof} 
Noting that $\normalof{\blowupLocusA}{ \ambient} \cong \buntot$ and $\normalof{\singular}{\ambient} \cong \bunproj^* \bunlin|_\singular$ so that $\normalof{\singular}{\ambient}|_\blowupLocusA \cong \bunlin$, the first isomorphism comes from Theorem~\ref{cor.blowupBgeom}(\ref{cor.blowupBgeomB}), and  the second isomorphism uses Lemma~\ref{lem.equiv}. Now (\ref{prop.bunspecialA}) follows using the definition of $\bunsec$.  For (\ref{prop.bunspecialB}), we again use Lemma~\ref{lem.equiv}.
\end{proof}

In the local setting we have the following strengthening of Proposition~\ref{prop.singloc}.

\begin{prop} In Setting~\ref{setn.local} the singular locus \,$\singLocus = \blowupLocusB \subset \blowupLocusA$.
\end{prop}
\begin{proof}
Note that $\singLocus \subset \blowupLocusA$ using the construction of $\singular$, and apply Proposition~\ref{prop.singloc}.
\end{proof}

\begin{prop}\label{prop.canonicallocal} In Setting~\ref{setn.local} we have
$\omega_{\ambient} \cong \bunproj^* (\omega_{\blowupLocusA} \otimes \det \buntot^\vee)$
 and furthermore
\beq
\omega_{\singular} \cong  \bunproj^* (\omega_{\blowupLocusA} \otimes \sHom(\det \buntot, \bunlin))  \cong  \bunproj^* (\omega_{\blowupLocusA} \otimes  \sHom(\bunlin, \det \bunsec))
\eeq
where we reuse the notation~$\bunproj$ for the restriction of~$\bunproj\colon\ambient \to \exbase$ to~$\singular$.
\end{prop}
\begin{proof} The claim for $\omega_{\ambient}$ is standard, using that $\ambient$ is a bundle over~$\blowupLocusA$. The claim for $\omega_{\singular}$ then uses adjunction $\omega_{\singular} \cong \omega_{\ambient}|_{\singular} \otimes \bunproj^*\bunlin$,
and the last isomorphism is by Lemma~\ref{lem.equiv}.
\end{proof}

I give examples of Setting~\ref{setn.local}, taking $\bunlin$ trivial unless stated otherwise. First I give $3$-folds $\singular$ which may be constructed from surfaces $\exbase$ giving $4$-folds $\ambient$.

\begin{eg}\label{eg.fourfold} Take $\exbase = \mathbb{P}^1 \times \mathbb{P}^1$ and a generic section $\sec$ of $\bunsec = \cO(1,1)^{ \oplus 2} $. For instance, taking coordinates $(x_i:y_i)$ with $i=1,2$ for the $\mathbb{P}^1$ factors, we may take $\sec = (x_1x_2,y_1y_2)$.  Then $\blowupLocusB$ is the points $((1:0),(0:1))$ and $((0:1),(1:0))$.
The singular 3-fold~$\singular$ is the zeroes of a function on a 4-fold\footnote{For further study of the derived category of this 4-fold $\ambient$, see \cite{Kite,Don}.} namely the total space~$\ambient$ of~$\buntot = \cO(-1,-1)^{ \oplus 2} $.
\end{eg}

\begin{eg} Take $ \exbase =\mathbb{P}^2$ and a generic section $\sec$ of $\bunsec$ its tangent sheaf. Then by a Chern class calculation $\sec$ has 3~zeroes, so that $\blowupLocusB$ is 3~points. The singular \mbox{3-fold} $\singular$ is then the zeroes of a function on  the total space~$\ambient$ of~$\buntot$, the {co}tangent sheaf of~$\mathbb{P}^2$.
\end{eg}

To describe $\tw[\hun]$ in these cases, we have the following.

\begin{prop} Take $\dim \singular= 3$ so that $\dim \blowupLocusB = 0$ with $\blowupLocusB$ given by reduced points $\{z_1,\dots,z_p\}$. Then 
$\tw[\hun] \cong  \twObj{\resolution}{\cE_1} \comp \dots \comp \twObj{\resolution}{\cE_p}$
where $\cE_i = \cO_{\excBlowupBshort_i}(-1)$ for $\excBlowupBshort_i=\resMap^{-1} (z_i)\cong\mathbb{P}^1$.
\end{prop}
\begin{proof} We have that $\hun\colon \cO_{z_i} \mapsto \cE_i$ and that the $\cE_i$ are orthogonal in $\D(\resolution)$ because their supports are disjoint. The claim follows by standard methods, compare Remark~\ref{rem.redconifold}.
\end{proof}

I now give examples with positive-dimensional $\blowupLocusB$ as follows.

\begin{eg}\label{eg.ellnormcurveetc} We  obtain singular Calabi--Yau hypersurfaces $\singular$ in 
$\ambient= \Tot \buntot\cong \Tot \bunsec^\vee$ by taking spaces as follows.
\begin{itemize}
\item Take $\exbase = \mathbb{P}^3$ and $\bunsec = \cO(2) ^{ \oplus 2}$. Then $\blowupLocusB$ is an elliptic normal curve of degree $4$.

\item  Take $\exbase = \mathbb{P}^4$ and $\bunsec = \cO(2)  \oplus \cO(3)$. This gives $\blowupLocusB$ a K3~surface of genus~4.

\item  Take $\exbase = \mathbb{P}^{2k+1}$ and $\bunsec = \cO(k+1) ^ {\oplus 2}$. Then $\blowupLocusB$ is a Calabi--Yau $(2k-1)$-fold.
\end{itemize}
\end{eg}

Generalizing some of the examples above, we may take $\exbase$ with a spin structure, in the sense that $\omega_\exbase$ has a square root, as follows.

\begin{eg} Take $\exbase$ with an invertible sheaf $\buntheta$ such that $ \buntheta \otimes \buntheta \cong \omega_\exbase $ and let $\bunsec = \buntheta ^ {{\vee}\oplus 2}$. Then $\ambient= \Tot \buntot$ with $\buntot \cong \buntheta  \otimes \field^2$, and $\ambient$ is Calabi--Yau using Proposition~\ref{prop.canonicallocal}.
\end{eg}

\begin{rem}\label{rem.dims}
In the above examples $\det \bunsec^\vee\cong\det \buntot\cong \omega_\exbase$, so $\ambient$ and $\singular$ are Calabi--Yau by Proposition~\ref{prop.canonicallocal}. The same then holds for  $\blowupLocusB$  by Proposition~\ref{prop.canonical} so that Theorem~\ref{keythm.blowupB} relates Calabi--Yau spaces $\blowupLocusB$, $\singular$, $\ambient$, with dimensions $d-2$,~$d$,~$d+1$, respectively.

Spaces $\singular$ which are \emph{not} Calabi--Yau may be obtained by keeping the same section $\sec$ of~$\bunsec$ as in the above examples, but allowing  the invertible sheaf~$\bunlin$ to be non-trivial, as can be seen from Proposition~\ref{prop.canonicallocal}. More generally, we may take an arbitrary regularly embedded~$\blowupLocusB$ of codimension~$2$ in  smooth equidimensional $\blowupLocusA$, cut out by a section of some locally free~$ \bunsec$. For $\buntot = \sHom(\bunsec,\bunlin)$ with arbitrary invertible $\bunlin$, we are then in Setting~\ref{setn.local} using Lemma~\ref{lem.equiv}.
\end{rem}




\opt{annalen}{
\stdskip
\noindent The author states that there is no conflict of interest, and that the manuscript has no associated data.
}

\end{document}